\documentclass[a4paper,reqno,12pt]{amsart}
	\usepackage[latin1]{inputenc}
    \usepackage[T1]{fontenc}
    \usepackage[english]{babel}
    \usepackage{appendix}
    \usepackage[english]{minitoc}
    \usepackage[nice]{nicefrac}

\usepackage[top=3cm, bottom=3cm, left=3cm, right=3cm]{geometry}
	\usepackage{amsmath}    \usepackage{amssymb}   \usepackage{amsbsy} %\usepackage{amsthm}
    \usepackage{latexsym, amsfonts}
    \usepackage{graphics}
    \usepackage{ulem}
    \usepackage{hhline}
    \usepackage{dsfont}
    \usepackage{mathrsfs}
    \usepackage{color}
    \usepackage{fancyhdr}
    \usepackage{rotating}
    \usepackage{fancybox}
    \usepackage{colortbl}
    \usepackage{pifont}
    \usepackage{setspace}
    \usepackage{enumerate}
    \usepackage{multicol}
    \usepackage{varioref}
    \usepackage{lmodern}
    \usepackage{textcomp}
    \usepackage{euscript}
    \usepackage[pdftex]{hyperref}
    
			\newtheorem{theorem}{Theorem}[section]
			\newtheorem{definition}[theorem]{Definition}
			
			\newtheorem{lemma}[theorem]{Lemma}
			\newtheorem{proposition}[theorem]{Proposition}
			
			\newtheorem{corollary}[theorem]{Corollary}
			\newtheorem{remark}[theorem]{Remark}

\usepackage{xcolor}
\newcommand{\C}{\mathbb C}    \newcommand{\R}{\mathbb R}   \newcommand{\Z}{\mathbb Z} 	 
\newcommand{\Hq}{\mathbb H} \newcommand{\Hqr}{\widetilde{\mathbb{H}}} \newcommand{\Sq}{\mathbb S}
\newcommand{\bq}{\overline{q} } \newcommand{\bp}{\overline{p} } \newcommand{\Cn}{C_j(I)} \newcommand{\Cj}{C_j(I)}
\newcommand{\Fc}{\mathcal{F}^{2}}
\newcommand{\Fcn}{\mathcal{F}^{2}_n}
\newcommand{\srn}{\mathcal{SR}_{n}}
\newcommand{\srm}{\mathcal{SR}_{m}}
\newcommand{\srdn}{\mathcal{SR}^{2}_{n}}
\newcommand{\srdun}{\mathcal{SR}^{2}_{1,n}}
\newcommand{\srddn}{\mathcal{SR}^{2}_{2,n}}
\newcommand{\srddk}{\mathcal{SR}^{2}_{2,k}}
\newcommand{\eigendn}{\mathcal{F}_{n}^{2}}
\newcommand{\sgn}{\operatorname{sgn}}
\newcommand{\norm}[1]{{\left\|{#1}\right\|}}    
   \newcommand{\scal}[1]{{\left\langle{#1}\right\rangle}}
    \newcommand{\bz}{\overline{z}}  \newcommand{\bw}{\overline{w}}
   
\newcommand{\Boxu}{\Box} \newcommand{\Boxd}{\Delta}

\begin{document}
\title[S-polyregular Bargmann spaces]{S-polyregular Bargmann spaces}
%: reproducing kernels, associated Segal--Bargmann transforms and spectral realization}

\thanks{The research work of A.G. was partially supported by a grant from the Simons Foundation.}

\author{A. Benahmadi}
\email{abdelhadi.benahmadi@gmail.com}
\author{A. El Hamyani}
\email{amalelhamyani@gmail.com}
\author{A. Ghanmi}
\email{allal.ghanmi@um5.ac.ma}

\address{\quad \newline A.E.D.P.G.S., CeReMAR, 
          Department of Mathematics, \newline P.O. Box 1014,  Faculty of Sciences, \newline
          Mohammed V University in Rabat, Morocco}

\begin{abstract}
We introduce two classes of right quaternionic Hilbert spaces in the context of slice polyregular functions, generalizing the so-called slice and full hyperholomorphic Bargmann  spaces.
Their basic properties are discussed, the explicit formulas of their reproducing kernels are given and associated Segal--Bargmann transforms are also introduced and studied.
The spectral description as special subspaces of $L^2$-eigenspaces of a second order differential operator involving the slice derivative is investigated.
\end{abstract}
% Slice regular function; S--polyregular function; Identity Principle for S--polyregular function; Star product for S--polyregular function; S--polyregular Bargmann spaces; Reproducing kernel; Laguerre polynomials; Sliced magnetic Laplacian; Kummer's function; $L^2$--eigenspace
\maketitle

\section{Introduction}

The classical Bargmann functional space $\Fc$ is defined as the phase space on the complex plane consisting of all $e^{-|z|^2} dxdy$-square integrable entire functions.
It is known to be unitarily isomorphic to the quantum mechanical configuration space $L^2(\R;dt)$ by means of the classical Segal--Bargmann transform (see for examples \cite{Bargmann1961,Folland1989,Zhu2012}).
As special generalizations, in the context of polyanalytic functions, are the generalized Bargmann spaces $\Fcn$ of level $n=0,1,2, \cdots$, (see for example \cite{AbreuBalazsGossonMouayn2015,AbreuFeichtinger2014,GI-JMP2005,Vasilevski2000}), so that $\mathcal{F}^{2}_0=\Fc$.
The corresponding theory has found remarkable applications in time-frequency analysis, analysis of the higher Landau levels and in the multiplexing of signals %(i.e., analysis of several signals simultaneously),
(see \cite{AbreuFeichtinger2014} and the references therein).

A quaternionic analogue $\mathcal{F}^{2}_{slice}$ of $\Fc$ was introduced in \cite{AlpayColomboSabadiniSalomon2014} involving special slice regular functions $\mathcal{SR}$ on the quaternion algebra $\Hq$, i.e., $\Hq$-valued real differentiable functions $f$ on $\Hq\equiv \R^4$ such that
\begin{align}\label{dbar}
\overline{\partial_I} f(x+Iy) := %= \overline{\partial_{s}}f|_{\C_I}(x+Iy) :=
\dfrac{1}{2}\left(\frac{\partial }{\partial x}+I\frac{\partial }{\partial y}\right)f|_{\C_I}(x+yI)
\end{align}
vanishes identically on $\Hq$ for every $I\in \Sq=\{q\in{\Hq};q^2=-1\}$. Above and hereafter, $f|_{\C_I}$ denotes the restriction of $f$ to the slice $\C_{I} := \mathbb{R}+\mathbb{R}I$ in $\Hq$. More precisely,
  \begin{align}\label{SHBFspace}
 \mathcal{F}^{2}_{slice} &= \left\{ f(q) =\sum_{j=0}^{+\infty} q^{j}  c_j; \, c_j\in \Hq
 , \quad
  \sum_{j=0}^{+\infty} j!|c_j|^{2}<+\infty \right\}.
 % \\  & \subset \mathcal{SR} \cap L^2(\C_I;e^{-|q|^2}d\lambda_I). \nonumber
   \end{align}
It is shown in \cite{AlpayColomboSabadiniSalomon2014} that $\mathcal{F}^{2}_{slice}$ is independent of $I$ and is a reproducing kernel quaternionic Hilbert space.
The related quaternionic Segal--Bargmann transform is studied in \cite{DG2017}. It connects $\mathcal{F}^{2}_{slice}$ to the $L^2$-Hilbert space of quaternionic-valued functions on the real line.

Motivated by the works \cite{AbreuBalazsGossonMouayn2015,AbreuFeichtinger2014,Balk1991,GI-JMP2005,Vasilevski2000} studying and characterizing the polyanaliticity in the complex setting as well as by Brackx' works \cite{Brackx1976a,Brackx1976b} studying the $k$-monogenic functions with respect to the Fueter operator, our aim in \cite{ElHamyaniG2017} was the study of possible generalizations of $\mathcal{F}^{2}_{slice}$ and its associated Segal--Bargmann transform to the context of slice $n$-polyregular ($\srn$) functions with respect to the slice derivative.
The concrete description of these spaces invoke the quaternionic Hermite polynomials
\begin{equation}\label{20}
H^Q_{m,n}(q,\overline{q})=m!n! \sum_{j=0}^{\min(m,n)}\dfrac{(-1)^{j}}{j!} \frac{q^{m-j}\overline{q}^{n-j}}{(m-j)!(n-j)!}
\end{equation}
for which we have provide an accurate systematic study in \cite{ElHamyani2018}. Such polynomials are the quaternionic analogues of the polyanalytic Hermite polynomials $H_{m,n}(z,\bz)$ (\cite{Gh13ITSF,Ito1952}) that play a crucial role in studying some basic properties of polyanalytic functions \cite{AbreuFeichtinger2014}.

The present paper is in fact an improved version of \cite{ElHamyaniG2017}. We consider two kinds of such generalizations. These spaces will be called here S-polyregular Bargmann space of level $n$ of first and second kind and we will denote them by $\srdun$ and $\srddn$, respectively.
	It should be noted that $\mathcal{SR}^2_{1,0} =\mathcal{SR}^2_{2,0}$ and reduce further to $ \mathcal{F}^{2}_{slice}$ in \eqref{SHBFspace} (see Section \ref{gbs1}).
Both $\srdun$ and $\srddn$ are natural extension of $\mathcal{F}^{2}_{slice}$ to the setting of S-polyregular functions (see Definition \ref{SPR}) and appear as special subspaces of
the Hilbert space
$$ \srdn := \srn \cap L^2(\C_I; e^{-|\xi|^2} d\lambda_I),$$
the space of all S-polyregular functions $f:\Hq \longrightarrow \Hq$ subject to the norm boundedness $\norm{f}_{\C_I}<+\infty$,
%$$ \int_{\C_I} |F|_{\C_I}(\xi)|^2 e^{-|\xi|^2} d\lambda(\xi) < +\infty.$$
where $\norm{\cdot}_{\C_I}$ is the norm induced by the inner product
\begin{equation}\label{spfg}
 \scal{f,g}_{\C_I} = \int_{\C_I}\overline{f|_{\C_I}(q)} g|_{\C_I}(q)e^{-\vert{q}\vert^2} d\lambda_I(q).
\end{equation}

Our main aim is to give a concrete description of these spaces.  We prove that $\srdun$ and $\srddn$ are reproducing kernel quaternionic Hilbert spaces whose reproducing kernels are given explicitly in terms of Laguerre polynomials (see Theorem \ref{thm:RepKernel}).
The proof is based essentially on a weak version of the Identity Principle for S-polyregular functions that we establish in Subsection 2.2 (Proposition \ref{SpolyIP}) and on a natural extension of the left star product for S-polyregular functions. Moreover, a hilbertian decomposition of $L^2(\C_I; e^{-|\xi|^2} d\lambda_I)$ in terms of $\srddn$ is also given (Theorem \ref{HilbDecomp}).

Associated Segal--Bargmann transforms $\mathcal{B}_{\ell,n}$, $\ell=1,2$, are then introduced and studied in some details  (see Theorems \ref{thm:SBTn} and \ref{thm:isometrySBTm}).
They are defined on $L^2_{\Hq}(\R;dt)$, the $L^2$-Hilbert space of left quaternionic-valued functions on the real line. Their kernel functions involve the Hermite polynomials extended to the quaternions.
It should be noted here that for $n=0$, the transform $\mathcal{B}_{\ell,0}$ is equal to the one considered in \cite{DG2017}.

Another task of the present paper is to show that the constructed spaces are closely connected to the concrete $L^2$-spectral analysis of the semi-elliptic (slice) second-order differential operator
\begin{equation}\label{Laplaceian}
\Boxu_{q}= - \partial_s \overline{\partial_s} + \bq \overline{\partial_s},
\end{equation}
where
\begin{equation}\label{120}
\overline{\partial_s} f(q)=\left\{
\begin{array}{ll}
\overline{\partial_{I_q}} f 
(x+I_q y)
, & \hbox{if } q=x+I_{q} y \in \Hq \setminus\R; \\
\dfrac{df}{dx}(x), & \hbox{if }  q=x\in  \R ,
\end{array}
\right.
\end{equation}
which can seen as the conjugate of the left slice derivative $\partial_{s}$ that we can define in a similar way in terms of $\partial_{I_q}$.
In fact, such spaces are realized as special subspaces of the $L^2$-eigenspaces
\begin{equation}\label{37r}
\eigendn= \left\{ f \in L^{2}(\Hq;e^{-\mid q\mid^{2}}d\lambda); \, \Boxu_{q}f=n f  \right\},
\end{equation}
where $n=0,1,2, \cdots$, and $d\lambda$ denotes the Lebesgue measure on $\Hq\simeq \R^4$.
The $L^2$-spectral description of $\Boxu_{q}f=n f $ was possible by dealing first with the $\mathcal{C}^\infty$ right-eigenvalue problem
$\Boxu_{q}f= f\mu $ on $\Hqr:=\mathbb{H}\setminus{\R}$ 
and then by extending appropriately the obtained explicit solutions to the whole $\Hq$ (Theorem \ref{thm:CInftydescription}).
Thereby, by manipulating the asymptotic behavior of such eigenfunctions, we show that the spectrum of $\Boxu_{q}$ is purely discrete and consists of the eigenvalues $\mu= n$ %(Landau levels)
which occur with infinite degeneracy (see Theorem \ref{thm:description}).
The spaces $\mathcal{SR}^2_{\ell,n}$, $\ell=1,2$, are then the specific subspaces of $\eigendn$ described by Theorem \ref{thm:description}.
 This becomes clear in the discussion provided in the last section.

Added to this introductory section, the rest of the paper is structured as follows. We devote Section 2 to some elementary and needed properties satisfied by the S-polyregular functions. We describe in Section 3 the
S-polyregular Bargmann spaces $\mathcal{SR}^2_{1,n}$ and $\mathcal{SR}^2_{2,n}$ and we give the explicit formulas for their reproducing kernels. The associated Segal--Bargmann transforms are introduced and studied in Section 4. While, we have reserved Section 5 to the spectral realization of these S-polyregular Bargmann spaces. Some concluding remarks on the full S-polyregular Bargmann spaces constitute the content of Section 6.

\section{S-polyregular functions}
%%%%%%%%%%%%%%%%%%%%%%%%%%%%%%%%%%%%%%%%%%%%%%%%%%%%%%%%%%%%%%%%%%%%%%%%%%%%%%%%%%%%%%%%%%%%%
\subsection{The real skew algebra of quaternions.} 
%%%%%%%%%%%%%%%%%%%%%%%%%%%%%%%%%%%%%%%%%%%%%%%%%%%%%%%%%%%%%%%%%%%%%%%%%%%%%%%%%%%%%%%%%%%%%

 The elements of the division algebra of quaternions $\Hq$ are $4$-component extended complex numbers of the form
 $q=x_0+x_1\mathbf{i}+x_2\mathbf{j}+x_3\mathbf{k} \in{\Hq}$, where  $x_0,x_1,x_2,x_3\in{\mathbb{R}}$ and the imaginary components $\mathbf{i}$, $\mathbf{j}$, $\mathbf{k}$ satisfy the Hamiltonian computation rules
 $\mathbf{i}^2=\mathbf{j}^2=\mathbf{k}^2=\mathbf{i}\mathbf{j}\mathbf{k}=-1$;
$\mathbf{k}\mathbf{i}=-\mathbf{i}\mathbf{k}=\mathbf{j}.$
   According to this algebraic representation, the quaternionic conjugate is defined to be $x_0-x_1i-x_2j-x_3k =\Re(q)-\Im (q)$, where $\Re(q)=x_0$ and $\Im (q)=x_1i+x_2j+x_3k$ . Here and elsewhere in the paper $\overline{q}$ denotes the algebraic conjugate of the quaternion $q\in\Hq$.
   Then, we have $\overline{ pq }= \overline{q}\, \overline{p}$ for $p,q\in \Hq$, and the modulus of $q$ is defined to be
 $|q|=\sqrt{q\overline{q}}$.  The  polar representation is given by
$ q=re^{I\theta},$ 
where $r=|q| \geq 0$, $\theta \in [0,2\pi[$,
and $I$ belongs to the set of imaginary units $\mathbb{S}$,
 which can be identified with the unit sphere $S^2=\lbrace{q\in{\Im \Hq}; \vert{\Im (q)}\vert=1}\rbrace$ in $\Im \Hq=\R \mathbf{i}+\R\mathbf{j}+\R\mathbf{k}$.
The representation $q=re^{I\theta}$ is not unique unless $q$ is not real.
Another interesting representation of $q\in \Hq$ is given by $q=x+I y$ for some real numbers $x$ and $y$ and imaginary unit $I\in \mathbb{S}$. It is unique for any $q\in \Hqr=\Hq\setminus \R$ with $y>0$. Thus, $\Hq$ can be seen as the infinite union of the slices $\C_{I} := \mathbb{R}+\mathbb{R}I$.
The last representation was crucial in developing the theory of quaternionic slice hyperholomorphic functions that has been introduced
by Gentili and Struppa in their seminal work \cite{GentiliStruppa07}. 
Since then, they  have been object of intensive research and the corresponding hypercomplex analysis has been developed. It has found many interesting applications in operator theory, quantum physics, Schur analysis and different branches of differential geometry. See for instance \cite{%AlpayColomboSabadini2012,
	Altavilla2018,AltavilladeFabritiis2018-1,% ColomboSabadiniStruppa2011,
	GentiliStoppatoStruppa2013,
	GentiliStruppa2010} and the references therein. 

%%%%%%%%%%%%%%%%%%%%%%%%%%%%%%%%%%%%%%%%%%%%%%%%%%%%%%%%%%%%%%%%%%%%%%%%%%%%%%%%%%%%%%%%%%%%%
\subsection{S-polyregular functions and first properties. }
%%%%%%%%%%%%%%%%%%%%%%%%%%%%%%%%%%%%%%%%%%%%%%%%%%%%%%%%%%%%%%%%%%%%%%%%%%%%%%%%%%%%%%%%%%%%%

The solution of the Cauchy--Riemann equation $\overline{\partial_I} f|_{\C_I}=0$ on $\Hq$, involving the derivative in \eqref{dbar}, leads to the left power series
\begin{align} \label{remSliceCoef}
\varphi(x+Iy) = \sum_{j=0}^{+\infty} (x+Iy)^j \alpha_j(I),
\end{align}
with infinite convergent radius, where $\alpha_j$ are seen as functions $\alpha_j:I \longmapsto \alpha_j(I)$ on $\Sq$ with values in $\Hq$. If in addition $\alpha_j(I)$ are constants on $\Sq$,  we recover the standard space of slice regular functions
\cite{%ColomboSabadiniStruppa2011,ColomboSabadiniStruppa2016,
	GentiliStoppatoStruppa2013,GentiliStruppa07}.
A natural generalization is that of S-polyregular functions.

\begin{definition}[\cite{ElHamyani2018,AlpayDikiSabadini2019-3}]\label{SPR}
	A quaternionic-valued function $f$ on a domain $\Omega \subset \Hq$ such that $\Omega\cap \R \ne \emptyset$ is said to be (left) slice polyregular (S-polyregular) of level $n$ (order $n+1$), if it is a real differentiable in $\Omega$ and its restriction $f_{\Omega_I}$ is polyanalytic in $\Omega_I:=\Omega\cap \C_I$ for every $I\in \mathbb{S}$, in the sense that
	the function $\overline{\partial_I}^{n+1} f : \Omega_I \longrightarrow \Hq$  vanishes identically on $\Omega_I$.
	We denote by $\mathcal{SR}_n(\Omega)$ the corresponding right quaternionic vector  space.
\end{definition}

Topologically, the space $\mathcal{SR}_n(\Omega)$ is endowed with the natural topology of uniform convergence on compact sets in $\Omega$, so that it turns out to be a right vector space over the non-commutative field $\Hq$. We provide below some of their basic properties that we need to develop the rest of the paper for the case $\Omega=\Hq$. Thus, one can easily prove the following elementary characterization for the elements in $\mathcal{SR}_n :=\mathcal{SR}_n(\Hq)$ in terms of the elements of $\mathcal{SR}$.

\begin{proposition}[\cite{AlpayDikiSabadini2019-3}]\label{propSPDecom}
	 For every $f\in \mathcal{SR}_n$, there exist some  $\varphi_k \in \mathcal{SR} $, $k=0,1,\cdots, n$, such that
	$$f(q,\overline{q})=\sum_{k=0}^{n}\overline{q}^k\varphi_k(q).$$
\end{proposition}

%\begin{proof} 
	Whose proof is immediate and lies essentially on the characterization of polyanalytic functions in complex setting \cite{AbreuFeichtinger2014,Balk1991}. 
%	In fact,  $f|_{\C_I}:\C_I \longrightarrow \C_I$ is a polyanalytic function of level $n$ if and only if it can be rewritten as
%	$f|_{\C_I}(z,\bz)= \sum\limits_{j=0}^n \bz^j \varphi_j(z),$
%	with $\varphi_j$ are entire functions on $\C_I$.
%	\end{proof}
The following result is a second characterization of S-polyregular functions.

\begin{theorem}\label{thmSPDecom2}
	A function $f $ belongs to $\mathcal{SR}_n$ if and only if there exists $\varphi_0 \in \mathcal{SR}$ such that
	$$f(q,\overline{q})=\varphi_0(q) + \sum_{j=1}^{n} \sum_{k=0}^{n-j}  (-1)^k \frac{\overline{q}^{j+k}}{j!k!} \overline{\partial_s}^{j+k}f(q).$$
\end{theorem}

\begin{proof}
		By Proposition \ref{propSPDecom}, any $f\in \mathcal{SR}_n$ is of the form
	$f(q,\overline{q})=\sum\limits_{k=0}^{n}\overline{q}^k\varphi_k(q)$ for some  $\varphi_k \in \mathcal{SR} $, $k=0,1,\cdots, n$.
	Therefore, $\overline{\partial_s}^k f =0$ whenever $k>n$, and
	$$ \overline{\partial_s}^k f
	=  \sum_{j=0}^n \overline{\partial_s}^k(\overline{q}^j) \varphi_j
	=    \sum_{j= k}^n \frac{j!}{(j-k)!} \overline{q}^{j-k} \varphi_j $$
	when $k\leq n$. By considering the particular cases $k=n$, $k=n-1$, $k=n-2$ and $k=n-3$, one proposes the following
	$$(n-k)! \varphi_{n-k} = \sum_{s=0}^{k}   (-1)^s \frac{\overline{q}^{s}}{s!} \overline{\partial_s}^{n-k+s} f$$
	for $k<n$, which can be proved by induction. Equivalently, we write
		\begin{align}\label{component}
		\varphi_{j} = \frac{1}{j!} \sum_{s=0}^{n-j}   (-1)^s \frac{\overline{q}^{s}}{s!} \overline{\partial_s}^{j+s} f(q); \quad j\geq 1.
		\end{align}
		Therefore, the expression of $f$ becomes
			$$f(q,\overline{q})=\varphi_0(q) + \sum_{j=1}^{n} \sum_{k=0}^{n-j}  (-1)^k \frac{\overline{q}^{j+k}}{j!k!} \overline{\partial_s}^{j+k} f(q).$$
\end{proof}

\begin{remark}
	The component functions in Proposition \ref{propSPDecom}, of a given S-polyregular function $f$, are given in terms of $f$ and its successive derivatives (see Equation \eqref{component}).
\end{remark}

Thanks to these characterizations (Proposition \ref{propSPDecom} and Theorem \ref{thmSPDecom2}) many interesting analytic properties of S-polyregular functions can be derived from their analogues of the slice regular functions. However, one must be careful since (as is the case for complex polyanalytic functions) several known properties for $\mathcal{SR}$ prove false when applied to $\mathcal{SR}_n$.
For example, S-polyregular functions may even vanish on an accumulation set. This is the case of $1-q\overline{q}$ which is a nonzero S-polyregular on $\mathbb{H}$ but vanishes on the closed set $\{q\in \mathbb{H}, |q|=1\} $.

Similarly to the complex setting, the first order differential operator $\partial_s -  \overline{q}$, will play a crucial rule in this theory.
By considering the differential transformation
$$[\mathcal{H}_{n} (F)](q) := (\partial_s- \overline{q})^n(F)(q),$$
one proves the following.

\begin{theorem}\label{HermiteOp0}
	Let $F$ be a given S-regular function. Then, the functions $\mathcal{H}_{n} (F)$, $n=0,1,2, \cdots$, are S-polyregular  and form an orthogonal system in $L^2(\C_I;e^{-|q|^2}d\lambda_I)$.
\end{theorem}

\begin{proof}
	Notice first that
	\begin{align}\label{spolBinom}
	\mathcal{H}_{n} (F) = \sum_{j=0}^n (-1)^j\binom{n}{j} \overline{q}^j \partial_s^{n-j} F.
	\end{align}
	Therefore, we have
	$$\overline{\partial}_s^n\mathcal{H}_{n} F =  (-1)^n n! F.$$
	Hence, $\mathcal{H}_{n}F$  is clearly S-polyregular of order $n$, for $F$ being slice regular. 	
	This is also clear from \eqref{spolBinom} according to Proposition \ref{propSPDecom}.
	Consequently, using the fact that $\partial_s -  \overline{q}$ is the formal adjoint operator of  $\overline{\partial}_s$ when acting on the Hilbert space $L^2(\C_I;e^{-|q|^2}d\lambda_I)$, one can
	prove that $(\partial_s- \overline{q})F$ is orthogonal to $F$ when $F$ is slice regular in $L^2(\C_I;e^{-|q|^2}d\lambda_I)$. More generally, if $n>m$, we have $\overline{\partial}_s^{n-m} (F)=0$ and therefore
	 $$ \scal{\mathcal{H}_{n} (F),\mathcal{H}_{m} (F)} =  \scal{F,\overline{\partial}_s^n\mathcal{H}_{m} (F)} = (-1)^m m! \scal{F,\overline{\partial}_s^{n-m} (F)} =0.$$
	 Thus, $\mathcal{H}_{n} (F)$, $n=0,1,2, \cdots$, form an orthogonal system in $L^2(\C_I;e^{-|q|^2}d\lambda_I)$.
\end{proof}

\begin{remark}  By specifying $F(q)=F_m(q)=q^m$, we recover the quaternionic Hermite polynomials $H^Q_{m,n}$. Indeed,
	$$[\mathcal{H}_{n} (F_m)](q) = (-1)^m e^{|q|^2} \partial^n_s(e^{-|q|^2}q^m)= H^Q_{m,n}(q,\overline{q}).$$
\end{remark}

\begin{theorem}\label{thmHermiteOp}  The following assertions hold true.
	\begin{enumerate}
		\item[(i)] The space $\mathcal{SR}_n^2:= \mathcal{SR}_n  \cap L^2(\C_I;e^{-|q|^2}d\lambda_I)$
is spanned by the polynomials $ H^Q_{j,n}$, $j=0,1,2,\cdots$. Moreover, we have
$$ \mathcal{SR}_n^2 =	\sum_{k=0}^n\mathcal{H}_{k} (\mathcal{SR}_0^2).$$

	\item[(ii)] A function $f$ belongs to $\mathcal{SR}_n^2\cap Ker(\Box_q-nId)$ if and only if there exists some $F\in \mathcal{SR}_0^2$ such that $f=\mathcal{H}_{n} (F)$.
		\end{enumerate}
\end{theorem}

\begin{proof}
Let $f \in \mathcal{SR}_n^2 $ and recall that $H^Q_{j,k}(x+Iy,x-Iy)$ is an orthogonal basis of $L^2(\C_I;e^{-|q|^2}d\lambda_I)$ (see \cite[Theorem 4.2]{ElHamyani2018}). Thus, we can expand $f|_{\C_I}$ as
$$ f|_{\C_I} (x+Iy)= \sum_{j=0}^\infty  \sum_{k=0}^\infty  H^Q_{j,k}(x+Iy,x-Iy)\alpha_{jk}(I) $$
for some quaternionic sequence $\alpha_{jk}(I) \in \Hq$ satisfying the growth condition
$$ \sum_{j=0}^\infty  \sum_{k=0}^\infty  j!k! |\alpha_{jk}(I) |^2 <+\infty. $$
Now since $f$ is a polynomial in $\bq$ of degree $n$ (for $f$ being in $\mathcal{SR}_n$), we conclude that $\alpha_{jk}(I)=0$ for every $k>n$, so that
\begin{align} \label{idVCI}
f|_{\C_I} = \sum_{k=0}^n  \sum_{j=0}^\infty  H^Q_{j,k}(q,\bq)\alpha_{jk}(I) .
\end{align}
Therefore,
\begin{align*}
 f|_{\C_I} & = \sum_{k=0}^n  \sum_{j=0}^\infty  \mathcal{H}_{k} (q^j) \alpha_{jk}(I)
 = \sum_{k=0}^n  \mathcal{H}_{k} \left( \sum_{j=0}^\infty  q^j \alpha_{jk}(I)\right)
 = \sum_{k=0}^n  \mathcal{H}_{k} ( F_k) ,
\end{align*}
where $F_k$ stands for $F_k= \sum\limits_{j=0}^\infty  q^j \alpha_{jk}(I)$, which clearly belongs to $\mathcal{SR}_0^2$.
This completes the proof of $(i)$. To prove $(ii)$, we need only to establish the "only if". Thus, for $f\in \mathcal{SR}_n^2\cap Ker(\Box_q-n)$, we assert that $\Box_q f= nf$ is equivalent to have
$$\sum_{k=0}^n  \sum_{j=0}^\infty  k H^Q_{j,k}(q,\bq)\alpha_{jk}(I) = \sum_{k=0}^n  \sum_{j=0}^\infty  n H^Q_{j,k}(q,\bq)\alpha_{jk}(I)$$
thanks to \eqref{idVCI} combined with $\Box_q H^Q_{j,k}(q,\bq) = k H^Q_{j,k}(q,\bq)$ (see \cite{ElHamyani2018}). By identification, we get $\alpha_{jk}(I)=0$ for every $k\ne n$.
This completes the proof.
	\end{proof}

The following result is a Splitting Lemma for the S-polyregular functions generalizing the standard one for the slice regular functions (see \cite{AlpayDikiSabadini2019-3}).

\begin{proposition}[Splitting lemma for S-polyregular functions]\label{SpLem}
	If $f$ is a S-polyregular function, then for every
	$I\in \mathbb{S}$, and every $J \in \mathbb{S}$ perpendicular to $I$, there are two polyanalytic functions $F$, $G$ : $ \mathbb{C}_I \longrightarrow \mathbb{C}_I$ such that for any $q = x + I y$, we have
	$$ f|_{\C_I} (q) = F(q) + G(q)J .$$
\end{proposition}

%\begin{proof} By means of Proposition \ref{propSPDecom}, we can write $f \in \srn$ as
%	$\displaystyle f =\sum_{k=0}^{n}\overline{q}^k\varphi_k(q)$
%	for some $\varphi_k \in \mathcal{SR}$.
%	Hence, one can apply the standard Splitting Lemma (\cite{ColomboSabadiniStruppa2016,GentiliStoppatoStruppa2013}) for the slice regular functions to each $\varphi_k$. Indeed, for every $I,J\in \mathbb{S}$ such that $I\perp J$ and every $k=0,1,2,\cdots,n$, there are two holomorphic functions $F_k,G_k: \C_I \longrightarrow \C_I$ such that $\varphi_{k_{|_{\C_I}}}=F_k+G_kJ$.
%		Finally, for $q=x+Iy$, we obtain
%		\begin{eqnarray}
%		f|_{\C_I}(q,\bq)&=&\sum_{k=0}^{n}\overline{q}^k\varphi_{k|_{\C_I}}(q) \nonumber\\
%		&=&\left( \sum_{k=0}^{m}\overline{q}^k F_k(q)\right)  + \left(\sum_{k=0}^{m}\overline{q}^k G_k(q) \right) J ;
%		\end{eqnarray}
%		which reads $f{|_{\C_I}}= F+ GJ$ with $F$ and $G$ are polyanalytic functions of level $n$ (order $n+1$) on $\C_I$.
%\end{proof}

\begin{remark} The proof of Proposition \eqref{SpLem} readily follows from Proposition \ref{propSPDecom} and the standard Splitting Lemma (\cite{%ColomboSabadiniStruppa2016,
		GentiliStoppatoStruppa2013}) for the slice regular functions applied to each component function $\varphi_k$.
It can also be handled using sliceness, as pointed out to us by one of the referees. In
	fact, each slice function $f$ on $\Hq$ (not necessarily regular) can be split as $f|_{\C_I} (x+Iy) = F(x+Iy)+G(x+Iy)J$,
	where $J \perp I$ (see e.g. \cite{GhiloniPerotti2011-5}). Then, $f$ is polyregular of order $n$ if and only if $F$ and $G$ are polyanalytic of order
	$n_F$ and $n_G$, respectively, with $n = \max \{n_F,n_G\}$.
\end{remark}

An analogue of the Identity Principle for the S-polyregular functions can also be obtained. To this end, we begin by recalling the standard one for the slice regular functions on slice domains.

 \begin{definition}[\cite{%ColomboSabadiniStruppa2016,
 		GentiliStoppatoStruppa2013}]
 	A domain $U \subset \Hq$ such that $U\cap{\mathbb{R}}\ne \emptyset$ is said to be slice, if for every arbitrary $I\in{\mathbb{S}}$ the set $U_I:=U\cap{L_I}$ is a domain of the complex plane $\C_I:=\mathbb{R}+\mathbb{R}I$.
 	\end{definition}

\begin{lemma}[\cite{%ColomboSabadiniStruppa2016,
		GentiliStoppatoStruppa2013}] \label{StandIP}
	Let $f : U \longrightarrow \mathbb{H}$ be a slice regular function on
	a slice domain $U$. Denote by $Z_f = \{q\in U ; f(q)= 0 \}$ the zero set of $f$. If there
	exists $I \in \mathbb{S}$ such that $\mathbb{C}_I \cap Z_f$ has an accumulation point, then $f\equiv 0$ on $U$.
\end{lemma}

This principle is no longer valid for S-polyregular functions as shown by the counterexample $1 -q\bq$.  However, we can provide a weak version of such uniqueness theorem.

\begin{proposition}[Identity Principle for S-polyregular functions]\label{SpolyIP}
	Let $f$ be a S-polyregular function in $\mathcal{SR}_n$ such that $f$ is identically zero on a subdomain $\Omega\subset\C_I$ for some $I\in\Sq$. Then  $f$ is identically zero on the whole $\Hq$.
\end{proposition}

\begin{proof}
	By Proposition \ref{propSPDecom}, we can write $ f\in \mathcal{SR}_n$ as  $\displaystyle f(q) =\sum_{k=0}^{n}\overline{q}^k\varphi_k(q)$ with $\varphi_k \in \mathcal{SR}$.
	Now, by the assumption that $f|_{\Omega} \equiv 0$ with  $\Omega$ is a subdomain of some slice $\C_I$, we obtain
	$$
	n! \varphi_n|_{\Omega}(x+Iy) = \overline{\partial_I}^{n}
    \left(\sum_{k=0}^{n}(x-Iy)^k\varphi_k|_{\Omega}(x+Iy)\right) \equiv 0 .
	$$
	Repeating this procedure, we conclude that $\varphi_k|_{\Omega} \equiv 0$ for every $k=n,n-1, \cdots,1,0$.
	Therefore, $\varphi_k \equiv 0$ on the whole $\Hq$ by Lemma \ref{StandIP}. This implies that $f\equiv 0$ on $\Hq$.
\end{proof}

\begin{remark}
	Although Proposition \ref{SpolyIP} is Theorem 3.8  in \cite{AlpayDikiSabadini2019-3}, the proof we provide here is different.
\end{remark}

\begin{remark}
	Other powerful uniqueness theorems as well as additional properties for the $S$-polyregular functions can be obtained. They will be the subject of a forthcoming investigation.
\end{remark}

\subsection{Star product for S-polyregular functions.}
The authors of \cite{AlpayDikiSabadini2019-3} have introduced a star product for S-polyregular functions (for completness) without further results on it. In the sequel, we will review this notion and establish some related results. To this end, recall first that the left $\star^L_s$-product for left slice regular functions is defined by
\begin{equation}\label{starProduct}
 (f\star^L_s g) (q) = \sum_{n=0}^\infty q^n \left(\sum_{k=0}^n a_k b_{n-k}\right)
 \end{equation}
for given convergent series $\displaystyle  f (q) = \sum_{n=0}^\infty q^n a_n$ and $\displaystyle  g (q) = \sum_{n=0}^\infty q^n b_n $ on $\Hq$. This is in fact the product of two formal series with coefficients in a ring \cite{Fliess1974}.
The performed series in \eqref{starProduct} is convergent on $\Hq$ and is a slice regular function \cite{GentiliStoppato08}.
This product is introduced to overcome the fact that the pointwise product of  left slice regular functions is not necessarily a left slice regular function, but it is a S-polyregular function under further assumptions (see \cite{Ghanmi2019} for details).
For interesting results on the left $\star^L_s$-product, one can refer to \cite{AltavilladeFabritiis2019-2,%ColomboSabadiniStruppa2016,
	GentiliStoppatoStruppa2013} and the references therein.  To solve analogue problem in the context of left S-polyregular functions, a natural extension of the $\star^L_s$-product can be defined by considering
\begin{equation}\label{starSpProductL}
 (f\star^L_{sp} g) (q,\bq) = \sum_{j=0,1, \cdots ,m \atop  k=0,1,\cdots,n} \bq^{j+k} (\varphi_j \star^L_s \psi_k) (q)
  \end{equation}
for given $\displaystyle  f (q,\bq) = \sum_{j=0}^m \bq^j \varphi_j(q) \in \srm$ and $\displaystyle  g (q,\bq) = \sum_{k=0}^n \bq^k \psi_k(q) \in \srn $.
We define in a similar way the right star product for right S-polyregular functions
$\displaystyle  f (q,\bq) = \sum_{j=0}^m \varphi_j(q) \bq^j $ and $\displaystyle  g (q,\bq) = \sum_{k=0}^n \psi_k(q) \bq^k$ as follows
\begin{equation}\label{starSpProductR}
 (f\star^R_{sp} g) (q,\bq) = \sum_{j=0,1, \cdots ,m \atop  k=0,1,\cdots,n} ( \varphi_j \star^R_s \psi_k) (q) q^{j+k} .
 \end{equation}
Thus, one can easily check the following

\begin{lemma}\label{commstarp}
	 For every $f\in \srm$ and $g \in \srn $, we have
	\begin{enumerate}
		\item[(i)] $\overline{f\star^L_{sp} g} =  \overline{g} \star^R_{sp} \overline{f}$, where $\overline{f}(q) = \overline{f(q)}$ denotes the algebraic conjugation.
		\item[(ii)] $f\star^L_{sp} g = g\star^L_{sp} f$ if the coefficients of any components slice regular functions $\varphi_j$ and $\psi_k$ commute.
	\end{enumerate}
\end{lemma}

\begin{proof}
	Assertion $(i)$ follows by taking the algebraic conjugate in \eqref{starSpProductL}
and next using the well-established fact $\overline{\varphi_j \star^L_s \psi_k}=\overline{\psi_k} \star^R_s \overline{\varphi_j}$ for slice regular functions $\varphi_j$ and $\psi_k$.
	The second assertion is immediate by comparing $f\star^L_{sp} g$ and $g\star^L_{sp} f$.
\end{proof}
	
	A characterization for two S-polyregular functions to commute with respect to the $\star^L_{sp}$-product can be obtained, generalizing the one given in
	\cite{AltavilladeFabritiis2019-2} for $\C_J$-preserving slice regular functions.

	\begin{definition}[\cite{ColomboGonzGlez-CervantesSabadini2012,AltavilladeFabritiis2019-2}]
		 Let  $J\in \Sq$. A slice regular function $\varphi$ is said to be $\C_J$-preserving if both $F$ and $G$, in its stem function
		 $\varphi =\mathcal{I}(F+i G)$, are $\C_J$-valued.
	\end{definition}

	\begin{definition}	
			A S-polyregular function $f(q,\bq) = \sum\limits_{k=0}^n \bq^k \varphi_k(q) $ is said to be $\C_J$-preserving, for given  $J\in \Sq$, if their components slice regular functions $\varphi_k$ are $\C_J$-preserving.
	\end{definition}

\begin{lemma}\label{CJpreserving}
If $f$ and $g$ are two S-polyregular $\C_J$-preserving functions for given $J\in \Sq$, then $f\star^L_{sp} g = g\star^L_{sp} f$.
\end{lemma}

\begin{proof} 	
	The proof follows by making use of the fact that for $\C_J$-preserving functions $\varphi$ and $\psi$, the $\star^L_{s}$-product satisfies $\varphi\star^L_{s} \psi= \psi\star^L_{s} \varphi$ (see \cite%[Proposition 2.10]
	{AltavilladeFabritiis2019-2}).
\end{proof}

As basic example of computation with such $\star^L_{sp}$-product, we explicit the one of the following function
$$ S_k(\bp,p;q,\bq) : = \left( |p-q|^2_{\stackrel{L}{{\star}_{sp}}} \right)^{k \star^L_{sp}},$$
with $|p-q|^2_{\stackrel{L}{{\star}_{sp}}}:= (p-q) \star^L_{sp} \overline{(p-q)}= h_q(p) \star^L_{sp} \overline{h_q(p)}$, where we have set $h_q(p)=p-q$. Namely, we assert the following.

 \begin{lemma}\label{Slaguerre1}
	For every $k=1,2,\cdots $, and $p,q\in\Hq$, we have
	\begin{align}\label{form1} S_k(\bp,p;q,\bq)  = \sum_{j=0}^k (-1)^j \binom{k}{j} \bp^{k-j} h_q^{k\star^L_{s}}(p) \bq^j .
	\end{align}
\end{lemma}

\begin{proof} The proof can be handled by induction. Indeed, direct computation shows that for $k=1,2$, we have
	$$ S_1(\bp,p;q,\bq) = \bp (p-q) - (p-q)\bq = \bp h_q(p) - h_q(p)\bq $$
and
	\begin{align*}
	S_2(\bp,p;q,\bq) &= \left( \bp h_q(p) - h_q(p)\bq\right) \star^L_{sp} \left( \bp h_q(p) - h_q(p)\bq\right) \\
&= \bp^2 h_q^{2\star^L_{s}}(p)  - \bp h_q^{2\star^L_{s}}(p) \bq - \bp h_q^{2\star^L_{s}}(p)\bq  + h_q^{2\star^L_{s}}(p)\bq^2 \\
	&= \bp^2 h_q^{2\star^L_{s}}(p) - 2 \bp h_q^{2\star^L_{s}}(p) \bq + h_q^{2\star^L_{s}}(p) \bq^2.
	\end{align*}
    Now, assume that \eqref{Slaguerre1} holds true for fixed $k$. Then, we have
		\begin{align*}
		S_{k+1}(\bp,p;q,\bq) &= S_k(\bp,p;q,\bq) \star^L_{sp} S_1(\bp,p;q,\bq)
		\\ &= \sum_{j=0}^k (-1)^j \binom{k}{j}  \left( \bp^{k-j} h_q^{k\star^L_{s}}(p) \bq^j \right)  \star^L_{sp} \left( \bp h_q(p) - h_q(p)\bq\right)
		\\ &= \sum_{j=0}^k (-1)^j \binom{k}{j}   \bp^{k+1-j} h_q^{(k+1)\star^L_{s}}(p) \bq^j
			\\ & \qquad - \sum_{j=0}^k (-1)^j \binom{k}{j}  \bp^{k-j} h_q^{(k+1)\star^L_{s}}(p) \bq^{j+1}.
				\end{align*}
	Making the change of indices in the second sum in the right-hand side and using the well-known identity $\binom{k}{j}  + \binom{k}{j-1}= \binom{k+1}{j}$, we get
	\begin{align*}
S_{k+1}(\bp,p;q,\bq) & =	 \bp^{k+1} h_q^{(k+1)\star^L_{s}}(p)+	(-1)^{k+1}  h_q^{(k+1)\star^L_{s}}(p) \bq^{k+1}
\\ & \qquad +\sum_{j=1}^{k} (-1)^{j}  \left( \binom{k}{j}   + \binom{k}{j-1} \right)  \bp^{k+1-j} h_q^{(k+1)\star^L_{s}}(p) \bq^{j}	
\\& =	\sum_{j=0}^{k+1} (-1)^{j} \binom{k+1}{j}  \bp^{k+1-j} h_q^{(k+1)\star^L_{s}}(p) \bq^{j} 		  .
	\end{align*}
	This competes the proof.
\end{proof}

 Accordingly, it is clear that the following assertions hold true.
	\begin{enumerate}
	\item[(i)] The function $p \longmapsto S_k(\bp,p;q,\bq) $ is left S-polyregular for every fixed $q$.
	\item[(ii)] The function $q \longmapsto S_k(\bp,p;q,\bq) $ is right S-polyregular for every fixed $p$.
	\item[(iii)] We have $ \overline{S_k(\bp,p;q,\bq)} = S_k(q,\bq;\bp,p)$ for every $p,q\in \Hq$.
\end{enumerate}

The next result concerns the function on $\Hq\times \Hq$ defined by
\begin{align}\label{Laguerrestar}
L_{\star n}^{(\gamma, S_1)} (p,q):= L_{\star n}^{(\gamma)}( S_1(\bp,p;q,\bq))=L_{\star n}^{(\gamma)}\left(|p-q|^2_{\stackrel{L}{{\star}_{sp}}}\right),
\end{align}
 where $ L_{\star n}^{(\gamma)}$ is essentially the generalized Laguerre polynomial $ L_{n}^{(\gamma)}$
 (\cite{Rainville71}) but with respect to the $\star^L_{sp}$-product. It will be used to obtain the explicit expression of the reproducing kernels for the S-polyregular Bargmann spaces (see Section \ref{gbs1}).

 \begin{lemma}\label{Slaguerre1}
	The function $L_{\star n}^{(\gamma, S_1)}$ in \eqref{Laguerrestar}
    satisfies the properties $(i)$, $(ii)$ and $(iii)$ above.
\end{lemma}

\begin{proof} The proof readily follows since  $ L_{\star n}^{(\gamma)}$ is a finite linear expansion of the functions $S_k(\bp,p;q,\bq)$ with real coefficients.
	More precisely, we have
	$$
	L_{\star n}^{(\gamma)}( S_1(\bp,p;q,\bq)) =\sum _{k=0}^{n} \frac{\Gamma(\gamma+n+1)}{\Gamma(n-k+1)\Gamma(\gamma+k+1)} \frac{(-1)^{k}}{k!} S_k(\bp,p;q,\bq), $$
	where $\Gamma$ denotes the classical gamma function.
\end{proof}

%%%%%%%%%%%%%%%%%%%%%%%%%%%%%%%%%%%%%%%%%%%%%%%%%%%%%%%%%%%%%%%%%%%%%%%%%%%%%%%%%%%%%%%%%%%%%

In the next section, we introduce two classes of infinite dimensional right quaternionic reproducing kernels Hilbert spaces that will be considered as the quaternionic analogue of complex polyanalytic Bargmann spaces.

\section{S-polyregular Bargmann spaces} \label{gbs1}

The well-known analytic Hilbert spaces on the complex plane have been generalized to various contexts such as the quaternion setting (see for example \cite{%AlpayColomboSabadini2013,
	AlpayColomboSabadiniSalomon2014,deFabritiisGentiliSarfatti2018,
	%deFabritiisGentiliSarfatti2019-4,
	GhiloniPerotti2011-5}). Thus, the idea of generalizing the true-polyanalytic Bargmann spaces (\cite{AbreuFeichtinger2014,GI-JMP2005,Vasilevski2000}) to the slice polyregular case is rather natural.
This is the aim of the present section. To this end, let $\srdun$ denote the space of all convergent series
$$
f(q,\bq) = \sum_{k=0}^n \sum_{j=0}^\infty \bq^k q^j \alpha_{jk}; \quad \alpha_{j,k}\in\Hq ,
$$
on $\Hq$, belonging to the right $\Hq$-vector space
$ \srdn : =  \mathcal{SR}_n \cap L^2(\C_{I_0}, e^{-|\xi|^2}d\lambda)$, for some $I_0\in \Sq$.

The particular case of $n=0$ gives rise to the slice Bargmann space $\mathcal{F}^{2}_{slice}$ considered in \cite{AlpayColomboSabadiniSalomon2014}, for which the monomials $e_m(q):=q^m$ constitute an orthogonal basis. In contrast to what one can think, the monomials $e_{j,k}(q,\bq):=q^j\bq^k$ does not form an orthogonal system in $\srdn$ as showed by
$$\scal{e_{j,0},e_{j+k,k}}_{\C_I}=\norm{e_{j+k}}^2_{\C_I}=\pi (j+k)!.$$
 Thus, motivated by Theorem \ref{thmHermiteOp}, we will make use of the univariate quaternionic Hermite polynomials $ H^Q_{j,k}$, instead of monomials $e_{j,k}$, to describe $\srdn$.

\begin{proposition} A function $f$ belongs to $\srdun$ if and only it can be expanded as follows
	$$ f(q,\bq) = \sum_{k=0}^n\sum_{j=0}^{+\infty} H^Q_{j,k}(q,\bq)\alpha_{j,k}$$
	for some quaternionic constants $\alpha_{j,k}$ satisfying the growth condition
	$$ \sum_{j=0}^{+\infty} j!|\alpha_{j,k}|^2 <+\infty$$
	for every $k=0,1,\cdots,n$.
\end{proposition}

\begin{proof}
	The direct implication follows making use of \cite[Proposition 3.8]{ElHamyani2018},
	expressing the monomials $\bq^kq^j$ in terms of $H^Q_{r,s}$,
	\begin{align}\label{linearize}
	q^{m}\bq^{n}&= m!n! \sum_{k=0}^{min(m,n)}  \frac{ H^Q_{m-k,n-k}(q,\bq ) }{k!(m-k)!(n-k)!} .
	\end{align}
	The orthogonality
     \begin{align}\label{2}
	\scal{H^Q_{m,n},H^Q_{j,k}}_{\C_I}= \pi m!n! \delta_{m,j}\delta_{n,k}
	\end{align}
	of $H^Q_{r,s}$ shows that the condition $\norm{f}_{\C_I}<+\infty$ becomes equivalent to
	\begin{align*} \norm{f}^2_{\C_I} &= \sum_{k,k'=0}^n\sum_{j,j'=0}^{+\infty} \overline{\alpha_{j,k}} \scal{H^Q_{j,k}, H^Q_{j',k'} }_{\C_I} \alpha_{j',k'}
	\\&= \sum_{k=0}^n\sum_{j=0}^{+\infty} |\alpha_{j,k}|^2 \norm{H^Q_{j,k}}^2.
	\end{align*}
	The argument for obtaining the inverse implication is Theorem \ref{thmHermiteOp}.
\end{proof}

\begin{definition} \label{SRun}
	The right quaternionic vector space $\srdun$, generalizing the slice hyperholomorphic Bargmann space $\mathcal{F}^{2}_{slice}$,
	is called S-polyregular Bargmann space of  first kind and level $n$.
\end{definition}

Another interesting subspace to deal with is the following
$$\srddk:= \left\{  \sum_{j=0}^{+\infty} H^Q_{j,k}(q,\bq) c_j ;  c_j \in \Hq, \quad \sum_{j=0}^{+\infty} j!|c_j|^2 <+\infty
\right\} .$$

\begin{definition} \label{SRdn}
	The right quaternionic vector space $\srddk$ is called here S-polyregular Bargmann space of second kind and (exact) level $k$.
\end{definition}

\begin{theorem}\label{SFntBasis}
	The spaces $\srdun$ and $\srddk$ are Hilbert spaces with orthogonal basis $\{H^Q_{j,k}; k=0,1,\cdots,n; \, j=0,1,\cdots\}$ and  $\{H^Q_{j,k}; , j=0,1,\cdots\}$, respectively. Moreover, we have
	\begin{equation}\label{decomp}
	\srdun = \bigoplus_{k=0}^n \srddk.
	\end{equation}
\end{theorem}

\begin{proof} As for $n=0$, it is not difficult to see that the considered spaces are closed subspaces of the Hilbert space $L^2(\C_{I};e^{-|q|^2}d\lambda_I)$, and therefore they are right quaternionic Hilbert spaces. Now, for fixed nonnegative integer $k$, the polynomials $H^Q_{j,k}$,  $j=0,1,2,\cdots$, form an orthogonal system with respect to the gaussian measure and generate $\srddk$.
	Their linear independence is equivalent to their completion.
	In fact, for a given $ g = \sum\limits_{j=0}^{+\infty} H^Q_{j,k}  c_j \in \srddk$, the condition that $\scal{f,H^Q_{\ell,k}}=0$, for every $\ell=0,1,2,\cdots$, implies
	that $c_\ell=0$  and therefore $g$ is identically zero on $\Hq$, for
	$	\scal{f,H^Q_{\ell,k}}_{\C_I} = \overline{c_j} \norm{ H^Q_{\ell,k}}^2_{\C_I} .$
	Thus, $\{H^Q_{j,k} , j=0,1,\cdots\}$ is an orthogonal basis of $\srddk$. The assertion that $\{H^Q_{j,k}, k=0,1,\cdots,n, \, j=0,1,\cdots\}$ form  an orthogonal basis of $\srdun$ follows in a similar way. It is also an immediate consequence of \eqref{decomp}.
	The decomposition \eqref{decomp} readily follows since for given $f\in \srdun$, we have $$f = \sum_{k=0}^n\sum_{j=0}^{+\infty} H^Q_{j,k} \alpha_{j,k} = \sum_{k=0}^ng_k ,$$
	where
	$$  g_k:= \sum_{j=0}^{+\infty} H^Q_{j,k} \alpha_{j,k}.$$
	Then, it is clear that $g_k \in \srddk$. In addition, the family $\{g_k, k=0,1,\cdots,n\}$ is orthogonal, since for $k\ne k'$, we have
	$$\scal{g_k,g_{k'}}_{\C_I} =\left( \sum_{j,j'=0}^{+\infty}
	\overline{\alpha_{j,k}}\alpha_{j',k'} \delta_{j,j'}  \norm{H^Q_{j,k}}^2_{\C_I}\right) \delta_{k,k'} =0.$$
	Accordingly, we have $\displaystyle\srdun = \bigoplus_{k=0}^n \srddk$. Moreover,
	\begin{align}\label{Pythagor} \norm{f}^2_{\C_I} = \sum_{k=0}^n\norm{g_k}^2_{\C_I} = \pi \sum_{k=0}^n \sum_{j=0}^{+\infty}
	j!k! |\alpha_{j,k}|^2 .
	\end{align}
\end{proof}

In order to show that the considered Hilbert spaces $\srdun$ and $\srddk$ possess reproducing kernels, we need the following.

\begin{lemma}
	For every fixed $q\in \Hq$, the evaluation map $\delta_{q}f=f(q,\bq)$ is a continuous linear form on the Hilbert spaces $\srdun$ and $\srddk$.
	Moreover, we have
	\begin{equation}\label{ineq}
	|f(q,\bq)|  \leq  \frac{1}{\sqrt{\pi}} e^{ \frac{|q|^2}{2}} \norm{f}_{\C_I} .
	\end{equation}
	for every $f\in \srdun$ and therefore for every   $f  \in\srddk$.
	%\supset  \srddk$ $k=0,1,\cdots,n$.
\end{lemma}

\begin{proof}
	Let $g\in \srddk$ such that  $g= \sum\limits_{j=0}^{+\infty}  H^Q_{j,k}  c_{j} $.
	Using the Cauchy-Schwartz inequality and the expression of the square norm of $g$, $\displaystyle \norm{g}^{2}_{\C_I}= \pi k! \sum_{j=0}^{+\infty} j!|c_{j}|^2$, we get
	\begin{equation}\label{inequality}
	|g(q,\bq)|
	\leq \left(\sum_{j=0}^{+\infty} \frac{| H^Q_{j,k}(q,\overline{q}) |^2}{\pi j!k!}\right)^{\frac{1}{2}} \norm{g}_{\C_I} .
	\end{equation}
	The series in the right-hand side of \eqref{inequality} is absolutely convergent on $B(0,r_0)$ for every fixed $r_0$ and is independent of $g$. This follows readily making use of
	the following upper bound (see \cite[Corollary 4.3]{ElHamyani2018}):
	\begin{align}\label{estimate}
	\left| H^Q_{n+k,n}(q,\bq )\right|\leq \dfrac{(n+k)!}{k!}\left| q\right|^{k}  e^{\frac{|q|^2}{2}}.
	\end{align}
	More explicitly, by means of \cite[Eq. (3.8)]{BenahmadiG2019itsf}, we have
	\begin{align}\label{sumAbs}
	\sum_{j=0}^{+\infty} \frac{| H^Q_{j,k}(q,\bq) |^2}{\pi j!k!} = \frac{ e^{ |q|^2 }}{\pi} .
	\end{align}
	This proves that
	\begin{equation}\label{Esimateg}
	|g(q,\bq)|
	\leq \frac{1}{\sqrt{\pi}}  e^{\frac{|q|^2 }{2}} \norm{g}_{\C_I} .
	\end{equation}
	Now, for
	$\displaystyle f\in \srdun$, we have $\displaystyle f=\sum_{k=0}^ng_k$ with
	$g_k \in \srddk$, and therefore, we obtain
	$$ |f(q,\bq)|^2\leq \sum_{k=0}^n|g_k(q,\bq)|^2
	\leq \sum_{k=0}^n \frac{1}{\pi} e^{|q|^2 } \norm{g_k}^2
	\leq \frac{1}{\pi} e^{|q|^2 } \norm{f}^2_{\C_I} $$
	by means of \eqref{Pythagor} and \eqref{Esimateg}. This completes the proof.
\end{proof}

The previous Lemma combined with the quaternionic version of the Riesz' representation theorem \cite[Theorem 1]{TobarMandic2014} ensures the existence of the reproducing kernels  for $\srdun$ and $\srddk$.
The next result gives their explicit expressions in terms of
the Laguerre polynomial $L^{(\gamma)}_{\star n}$ and the special convergent series
$$e_{*}^{[a,b]}:=\sum_{n=0}^{+\infty}\dfrac{a^{k}b^{k}}{k!}.$$

\begin{theorem}\label{thm:RepKernel}
	The reproducing kernels of $\srdun$ and $\srddk$ are given respectively by
	\begin{align}\label{ExpRepKer1}
	\mathcal{K}_{1,n}(p,q)&= \dfrac{1}{\pi} e_{*}^{[\overline{p},q]} {\quad\stackrel{L}{{\star}_{sp}}\quad} L^{(1)}_{\star n}(|p - q|_{\stackrel{L}{{\star}_{sp}}}^2 )
	\end{align}
	and
	\begin{align}\label{ExpRepKer2}
	\mathcal{K}_{2,k}(p,q)&= \dfrac{1}{\pi}e_{*}^{[\overline{p},q]} {\quad\stackrel{L}{{\star}_{sp}}\quad} L_{\star k}(|p - q|_{\stackrel{L}{{\star}_{sp}}}^2 ) .
	\end{align} 	
	%Succinctly,
	Shortly, we have
	$$f(p,\bp) = \scal{\mathcal{K}_{1,n}(p,\cdot),f}_{\C_I}  \quad \mbox{ and } \quad g(p,\bp)= \scal{\mathcal{K}_{2,k}(p,\cdot),g}_{\C_I} $$
	for every $f\in \srdun$ and $g\in \srddk$.
\end{theorem}

\begin{proof}
	For $\srddk$, the computation of $\mathcal{K}_{2,k}(p,q)$ can be done by performing
	$$\mathcal{K}_{2,k}(p,q)=\dfrac{1}{\pi k!}\sum_{j=0}^{+\infty}\dfrac{H^Q_{j,k}(q,\overline{q})H^Q_{k,j}(p,\overline{p})}{j!},$$
	since  $\{H^Q_{j,k}; j=0,1,\cdots\}$ is an orthogonal basis of $\srddk$ (see Theorem \ref{SFntBasis}) and $\overline{H^Q_{k,j}} =H^Q_{j,k}$.
	For real $q=x$ or for $p,q$ belonging to the same slice $\C_I$, the result follows by means of
	\begin{align}\dfrac{1}{\pi j!}\sum\limits_{k=0}^{+\infty} \frac{ H^Q_{j,k}(z,\bz ) H^Q_{k,j}(w,\bw ) }{ k!  }   & =   \dfrac{(-1)^j}{\pi j!} e^{\bz w}   H^Q_{j,j}( z -w, \bz - \bw) \nonumber
	\\&=   \dfrac{1}{\pi}e^{ \bz  w }   L_j(|z -w|^2) \label{sameslice}
	\end{align}
	which is an immediate consequence of Theorem 2.3 in \cite{BenahmadiG2019itsf}, stating that
	$$
	\sum\limits_{j=0}^{+\infty} \frac{ t^j }{j!}  H_{k,j}(z,\bz ) H_{j,k'}(w,\bw )  =
	(-t)^{k'}   H_{kk'} ( z -tw, \bz - \overline{t}\bw) e^{ t \bz w}
	$$
	is true for every $|t|=1$ and $z,w\in \C$, %. % In particular, for $=t=1$ and $m=m'$, we get
	combined with the fact that $ H^Q_{k,k}(\xi,\bar\xi) =  (-1)^k k! L_k(|\xi|^2)$, where $L_k=L_k^{(1)}$ is the classical Laguerre polynomial of degree $k$.
	
	Now, for given fixed non-real $q$, let $I_q$ be in $\Sq$ such that $q\in \C_{I_q}$. The functions
	\begin{align}\label{fcts1} \varphi: p\longmapsto \mathcal{K}_{2,k}(p,q)
	\end{align}
	and
		\begin{align}\label{fcts2}  \psi: p\longmapsto
	\dfrac{1}{\pi}  e_{*}^{[\overline{p},q]}  {\quad\stackrel{L}{{\star}_{sp}}\quad} L_{\star k}(|p - q|_{\stackrel{L}{{\star}_{sp}}}^2 )
	\end{align}
	are clearly S-polyregular by the definition of the ${\star}_{sp}$-product (see \eqref{starSpProductL}) and Lemma \ref{Slaguerre1}. Moreover, they coincide on the slice $\C_{I_q}$ by means of \eqref{sameslice}.
Thus, by invoking the Identity Principle for S-polyregular functions (Proposition \ref{SpolyIP}), we conclude that $\phi\equiv \psi$ on the whole $\Hq$. This holds for arbitrary $q\in\Hq$.
	Therefore, we have
	\begin{align*}
	\mathcal{K}_{2,k}(p,q) =  \dfrac{1}{\pi}e_{*}^{[\bp,q]} {\quad\stackrel{L}{{\star}_{sp}}\quad} L_{\star k}(|p - q|_{\stackrel{L}{{\star}_{sp}}}^2 )
	\end{align*}
	for $p,q\in \Hq$. This completes our check for \eqref{ExpRepKer2}.
	
	To conclude for Theorem \ref{thm:RepKernel}, it suffices to observe that since $\displaystyle\srdun = \bigoplus_{k=0}^n \srddk$, we have
	$$ \mathcal{K}_{1,n}(p,q) = \sum_{k=0}^n  \mathcal{K}_{2,k}(p,q) .$$
	Hence, in virtue of $\displaystyle \sum_{k=0}^n  L^{(\gamma)}_k(x)= L^{(\gamma+1)}_n(x)$ (see \cite[Eq. 12, p. 203]{Rainville71}), we get
	\begin{align*}
	\mathcal{K}_{1,n}(p,q) &= \sum_{k=0}^n  \dfrac{1}{\pi}e_{*}^{[\overline{p},q]}  {\quad\stackrel{L}{{\star}_{sp}}\quad} L_{\star k}(|p - q|_{\stackrel{L}{{\star}_{sp}}}^2 )
\\&	=\dfrac{1}{\pi}e_{*}^{[\overline{p},q]} {\quad\stackrel{L}{{\star}_{sp}}\quad} L^{(1)}_{\star n}(|p - q|_{\stackrel{L}{{\star}_{sp}}}^2 ) .
	\end{align*} 	
\end{proof}

	\begin{remark} 	The restriction of $\mathcal{K}_{2,k}$ to $\C_i\times \C_i$ coincides with the reproducing kernel of the true-polyanalytic Bargmann space \cite{AIM2000JMP,BenahmadiG2019itsf,GI-JMP2005}. Indeed, we have
	\begin{align*}
	\mathcal{K}_{2,k}|_{\C_i\times \C_i}(z,w) =  \dfrac{1}{\pi}e^{\bz w}  L_{k}(|z-w|^2),
	\end{align*}
	so that for $k=0$, we recover the one of the classical Bargmann space $\displaystyle \dfrac{1}{\pi}e^{\bz w} $.
\end{remark}

\begin{remark} The expression of $\mathcal{K}_{1,n}(p,q)$ can be rewritten in the equivalent form
		\begin{align}\label{ExpRepKer1equiv}
	\mathcal{K}_{1,n}(p,q)&= \dfrac{1}{\pi}    L^{(1)}_{\star n}(|p - q|_{\stackrel{L}{{\star}_{sp}}}^2 ) {\quad\stackrel{L}{{\star}_{sp}}\quad} e_{*}^{[\overline{p},q]},
	\end{align}
	thanks to $(ii)$ of Lemma \ref{commstarp}. The same observation holds true for
	$\mathcal{K}_{2,k}(p,q)$.
\end{remark}

\begin{remark} The operator $f \longmapsto P_kf$
	given by $P_kf(p,\bp) = \scal{\mathcal{K}_{2,k}(p,\cdot) , f}_{\C_I}$, defined on the whole $\Hq$, defines a sort of extended orthogonal projection of $L^{2}(\C_I;e^{-|q|^{2}}d\lambda_I)$ onto $\srddk$. More explicitly, it reads
	\begin{align}
	P_kf(p,\bp) &= \dfrac{1}{\pi} \int_{\C_I} \overline{e_{*}^{[\bp,q]} {\quad \stackrel{L}{{\star}_{sp}}\quad } L_{\star k}(|p - q|_{ \stackrel{L}{{\star}_{sp}}}^2 ) } f(q,\bq) e^{-|q|^2}d\lambda_I(q) \label{OrthProj}
	\end{align}
for arbitrary $p\in \Hq$, which we can rewrite also as
		\begin{align}
	P_kf(p,\bp) &= \dfrac{1}{\pi} \int_{\C_I}   L_{\star k}(|p - q|_{\stackrel{R}{{\star}_{sp}}}^2 )  {\quad \stackrel{R}{{\star}_{sp}}\quad } \overline{e_{*}^{[\bq,p]}} f(q,\bq) e^{-|q|^2}d\lambda_I(q) \label{OrthProj}
	\end{align}
	by means of $(ii)$ in Lemma \ref{commstarp}.
	\end{remark}

We conclude this section with the following result giving an orthogonal hilbertian decomposition of the Hilbert space $L^{2}(\C_I;e^{-|q|^{2}}d\lambda_I)$.

\begin{theorem}\label{HilbDecomp}
	We have the following hilbertian decomposition
	\begin{align*}
	L^{2}(\C_I;e^{-|q|^{2}}d\lambda_I)=\bigoplus_{k= 0}^{+\infty} \srddk .
	\end{align*}
\end{theorem}

\begin{proof} Notice first that such decomposition is equivalent to the orthogonal complement of $\displaystyle \bigoplus\limits_{k\geq 0} \srddk$
	in $L^{2}(\C_I;e^{-|q|^{2}}d\lambda_I)$ is $\{0\}$. To this end, we claim that
	\begin{align}\label{iddeltasrk}
	T(t|q):=\int_{\C_I}  \dfrac{1}{(1-t)} \overline{e_{*}^{[\bq,\xi]} {\quad\stackrel{L}{{\star}_{sp}}\quad} \exp\left(-\dfrac{t}{1-t} |q-\xi|_{\stackrel{L}{{\star}_{sp}}}^{2}\right) } f(\xi) e^{-|\xi|^{2}}d\lambda(\xi)=0
	\end{align}
	holds for every $\displaystyle f\in \left(\bigoplus\limits_{k\geq 0} \srddk\right)^\perp$, every  $t\in ]0,1[$ and every fixed $q\in \Hq$.
	In fact, this follows readily making use of the generating function for the Laguerre polynomials (\cite[Eq. (14), p. 135]{Rainville71})
	\begin{align*}
	\sum_{k=0}^\infty  t^{k} L^{(\alpha)}_{k}(\xi) = \frac{1}{(1-t)^{\alpha +1}} \exp\left(\frac{t\xi}{t - 1} \right).
	\end{align*}
	Indeed,
	\begin{align*}
	T(t|q) & =  \int_{\C_I} \overline{e_{*}^{[\bq,\xi]} {\quad\stackrel{L}{{\star}_{sp}}\quad} \left(\sum_{k=0}^{+\infty}  t^{k} L_{\star k}(|w-q|_{\stackrel{L}{{\star}_{sp}}}^{2}) \right)} f(\xi)e^{-| \xi|^{2}}d\lambda(\xi)
	\\ & =	\sum_{k=0}^{\infty} t^{k} \int_{\C_I}  \overline{ e_{*}^{[\bq,\xi]} {\quad\stackrel{L}{{\star}_{sp}}\quad}  L_{\star k}(|\xi-q|_{\stackrel{L}{{\star}_{sp}}}^{2})e^{-|\xi|^{2}}} f(\xi) d\lambda(\xi)
	\\& =	0.
	\end{align*}
	The limit $t\longrightarrow 1^{-}$ in \eqref{iddeltasrk} yields an integral involving the Dirac $\delta$-function at the point $q\in \Hq$.
	From that, the left-hand side of \eqref{iddeltasrk} reduces further to $  e_{*}^{[\bq,\xi]} f(\xi)e^{-|\xi|^{2}}$. Therefore, $f(q)=0$ for every $q\in \Hq$.
\end{proof}

\section{Segal--Bargmann transforms for S-polyregular Bargmann spaces}

In this section, we introduce a family of suitable Bargmann's type transforms defined on the right quaternionic Hilbert space $L^{2}_{\Hq}(\R;dt)$, consisting of all square integrable $\Hq$-valued functions with respect to the inner product
\begin{align*}
\scal{ f,g }_{\R}:=\int_{\R} \overline{f(t)} g(t) dt.
\end{align*}
Their images will be the S-polyregular Bargmann spaces defined and studied in the previous section.
To this end, we define the kernel functions $ B_{\ell,n}(x;q)$, $\ell=1,2$, on $\R\times \Hq$ to be the bilinear generating functions
\begin{align}\label{kernel}
B_{2,k}(x;q) &= \sum_{j=0}^{+\infty}\dfrac{h_{j}(t)\overline{H^Q_{j,k}(q,\bq)}}{\norm{h_{j}}_{\R}\norm{H^Q_{j,k}}_{\C_I}} ,
\end{align}
and
\begin{align}\label{kernel2}
B_{1,n}(x;q) &= \sum_{k=0}^{n} B_{2,k}(x;q) ,
\end{align}
where $h_{j}(t)$ denotes the $j$-{th} real Hermite function \cite{Rainville71} associated to the real Hermite polynomial $H_{j}(t)$,
\begin{align}\label{6}
h_{j}(t)=(-1)^{j}e^{\frac{t^{2}}{2}}\dfrac{d^{j}}{dt^{j}}(e^{-t^{2}}) = e^{-\frac{t^{2}}{2}} H_{j}(t).
\end{align}
We recall that they form an orthogonal basis of $L^2_{\Hq}(\R;dt)$, with square norm given by
\begin{align}\label{1}
\norm{h_{j}}^{2}_{\R}=2^{j}j!\sqrt{\pi}.
\end{align}
Thus, we have

\begin{theorem} \label{thm:SBTn}
	For every $t\in \R$ and $q \in \Hq$, we have
	\begin{align} \label{kernelB2}
	 B_{2,k}(t;q)&:= \left(\dfrac{1}{\pi}\right)^{\frac 34}  \frac{1}{\sqrt{2^{k} k!}}
	\exp\left(-\frac{t^{2}+\bq^2}{2} +\sqrt{2}\bq t\right) H_{k}\left(\frac{q+\bq}{\sqrt{2}}-t\right).
	\end{align}
	Moreover, the function $ B_{2,k;q}:t\longmapsto  B_{2,k}(t;q)$
	belongs to $L^2_{\Hq}(\R;dt)$ for every  fixed $q\in \Hq$, and we have
	\begin{align}\label{5}
	\norm{  B_{2,k;q}}_{\R}=\dfrac{1}{\sqrt{\pi}}e^{\frac{|q|^{2}}{2}}.
	\end{align}
\end{theorem}

\begin{proof}
	The explicit expression of the kernel function $ B_{2,k}(t;q)$ can be obtained by \cite[Theorem 5.7]{ElHamyani2018}.
For the second assertion, fix $q=x+Iy$ in $\Hq$ and write the modulus of the kernel function $ B_{2,k}(t;q)$ as
\begin{align*}
|  B_{2,k}(t;q)|^{2}&=\left(\dfrac{1}{\pi}\right)^{\frac 32}  \frac{1}{2^{k}k!}
\left|e^{-\frac{t^{2}}{2}-\frac{x^{2}}{2}+\frac{y^{2}}{2}+Ixy+\sqrt{2} t\bq }\right|^{2}
\left|H_{k}(\sqrt{2}  x -t)\right|^{2}\\
&=\left(\dfrac{1}{\pi}\right)^{\frac 32}  \frac{1}{2^{k}k!} e^{-t^{2}-x^{2}+y^{2}+2\sqrt{2}xt} \left|H_{k}(\sqrt{2}x-t)\right|^{2}.
\end{align*}
Therefore, it follows that
\begin{align*}
\norm{B_{2,k;q}}^{2}_{\R}&=\left(\dfrac{1}{\pi}\right)^{\frac 32}  \frac{1}{2^{k}k!} e^{x^{2}+y^{2}} \int_{\R}e^{-(t-\sqrt{2}x)^{2}}|H_{k}(t-\sqrt{2}x)|^{2}dt
\\&=\left(\dfrac{1}{\pi}\right)^{\frac 32}  \frac{1}{2^{k}k!} e^{|q|^{2}}\int_{\R}e^{-u^{2}}| H_{k}(u)|^{2}du
\\&=\left(\dfrac{1}{\pi}\right)^{\frac 32}  \frac{1}{2^{k}k!} e^{|q|^{2}}\int_{\R} | h_{k}(u)|^{2}du
\\&=\left(\dfrac{1}{\pi}\right)^{\frac 32}  \frac{1}{2^{k}k!} e^{|q|^{2}}\norm{h_{k}}^2_{\R}
\\&= \frac{1}{\pi}e^{|q|^{2}}.
\end{align*}
\end{proof}

\begin{remark}
By comparing \eqref{5} to \eqref{sumAbs}, we conclude that $\norm{ B_{2,k;q}}_{\R}= \sqrt{K_k(q,q)}$ for every $q\in\Hq$.
\end{remark}

Associated to the kernel function $ B_{2,k}$ given through \eqref{kernel} (or also \eqref{kernelB2}), we are able to introduce a unitary integral transform (of Bargmann type) mapping isometrically the configuration space $L^2_{\Hq}(\R;dt)$ onto the constructed S-polyregular Bargmann space $\srddk$.
In fact, we have to consider 
$$[\mathcal{B}_{2,k}\phi](q):= \scal{B_{2,k}(\cdot;q),\phi}_{\R} .$$
 More explicitly,
\begin{align*}
[\mathcal{B}_{2,k}\phi](q)&:  = \left(\dfrac{1}{\pi}\right)^{\frac 34}  \frac{1}{\sqrt{2^{k} k!}}
\int_{\R}e^{-\frac{t^{2}+q^{2}}{2}+\sqrt{2} q t}
H_{k}\left(\frac{q+\overline{q}}{\sqrt{2}}-t\right)\phi(t)dt
\end{align*}
for a given function  $\phi:\R\rightarrow \Hq$, provided that the integral exists.
The following result shows that $\mathcal{B}_{2,k}$ is well-defined on $L^2_{\Hq}(\R;dt)$. Namely, we have

\begin{lemma} \label{thm:isometrySBTm}
For every quaternion $q\in \Hq$ and every $\phi \in L^2_{\Hq}(\R;dt) $, we have
\begin{align*}
\left| [\mathcal{B}_{2,k}\phi](q,\bq)\right|  &\leq \dfrac{1}{\sqrt{\pi}}e^{\frac{|q|^{2}}{2}}\norm{ \phi }_{\R}.
\end{align*}
\end{lemma}

\begin{proof}
The proof readily follows by applying the Cauchy-Schwartz inequality. In fact, we obtain
\begin{align}\label{4}
\left| \mathcal{B}_{2,k}\phi(q,\bq)\right|  &\leq  \int_{\R}| B_{2,k}(t;q)|  |\phi(t)| dt\leq \norm{ B_{2,k;q}}_{\R} \norm{\phi}_{\R}.
\end{align}
In view of \eqref{5}, the inequality \eqref{4} reduces further to
\begin{align*}
\left| \mathcal{B}_{2,k}\phi(q,\bq)\right|  &\leq \dfrac{e^{ \frac{|q|^{2}}{2}}}{\sqrt{\pi}} \norm{ \phi }_{\R}.
\end{align*}
\end{proof}

\begin{theorem} \label{thm:isometrySBTm}
The transform $\mathcal{B}_{2,k}$ defines a  Hilbert space isomorphism from $L^2_{\Hq}(\R;dt)$ onto $\srddk$.
\end{theorem}

\begin{proof}
	Notice first that the Segal--Bargmann transform $\mathcal{B}_{2,k}$
	maps the orthogonal basis $h_{j}$ of $L^2_{\Hq}(\R;dt)$ to the orthogonal basis $H^Q_{j,k}(q,\bq)$ of the S-polyregular Bargmann space $\srddk$. More precisely, we have
	$$[\mathcal{B}_{2,k}(h_{j})](q,\bq)= \left( \dfrac{1}{\pi}\right)^{\frac 14} \frac{\sqrt{2^j}}{\sqrt{k!}} H^Q_{j,k}(q,\bq).$$
	Then, one can conclude since $\mathcal{B}_{2,k}$ is continuous (by Lemma \ref{thm:isometrySBTm}).
\end{proof}

\section{Spectral realization of the S-polyregular Bargmann spaces}

\subsection{Discussion.}
In this section, we show that the S-polyregular Bargmann space $\srddn$ (and therefore $\srdun$) is closely connected to the concrete $L^2$-spectral analysis of the slice differential operator $\Boxu_{q}$ in \eqref{Laplaceian}.
To this end, we begin by considering the $\mathcal{C}^\infty$-spectral properties of $\Boxu_{q}$ which requires to solve two problems.
The first one is connected to the uniqueness problem of the polar representation $q=re^{I\theta}$ of the slice representation $q=x+I y$, of given $q\in \Hq$.
This can be resolved by restricting $q$ to $\Hqr=\Hq\setminus \R$ and next extend, somehow, the obtained results to the whole $\Hq$.
%This gives rise to the S-polyregular Bargmann spaces that will generalize the slice hyperholomorphic Bargmann  space \eqref{SHBFspace}.
The second problem is related to the notion of the slice derivative given by \eqref{120} which makes $\Boxu_{q}$ not necessarily elliptic.
To see this, notice that $\partial_s$ can be rewritten in the following unified form
\begin{equation}\label{sliceerUnified}
\partial_s = \frac{1}{2}  \left(\left(1+\chi_{\R}(q)\right)\frac{\partial}{\partial x}
- \left(1-\chi_{\R}(q)\right) I_q \frac{\partial}{\partial y}\right),
\end{equation}
so that the operator $\Boxu_{q}$  reads
\begin{align}\label{LaprealcoordHUnified}
\Boxu_{q} = -\frac{1}{4}& \left\{\left(1+\chi_{\R}(q)\right)^2 \frac{\partial^2}{\partial x^2}
+ \left(1-\chi_{\R}(q)\right)^2 \frac{\partial^2}{\partial y^2} \right\}
\\& +\frac{1}{2} \left(1+\chi_{\R}(q)\right) \left(x\frac{\partial}{\partial x} + y\frac{\partial}{\partial y} \right)
+\frac{I_q}{2}  \left(1-\chi_{\R}(q)\right)\left( x\frac{\partial}{\partial y} - y\frac{\partial}{\partial x} \right). \nonumber
\end{align}
It can be seen as a family of second order differential operators on $\R^2$ labeled by $\Sq$. Accordingly, for every fixed $I_q\in\Sq$, the operator $\Boxu_{q}$ is not elliptic nor uniform elliptic.
However, it is semi-elliptic  % \cite[Section 9]{OksendalBernt2003}.
since the eigenvalues of the corresponding matrix
$$  -\frac{1}{4} \left( \begin{array}{cc} \left(1+\chi_{\R}(q)\right)^2  & 0 \\ 0 & \left(1-\chi_{\R}(q)\right)^2  \end{array}\right)  $$
are clearly non-positives (but not necessary negatives).

Accordingly, a spectral realization of the S-polyregular Bargmann spaces, introduced in Section \ref{gbs1}, can be provided. To this end,  
we begin by studying the right eigenvalue problem of $\Boxu_{q}$ in \eqref{LaprealcoordHUnified}
when acting on the $\mathcal{C}^\infty$- as well as on the $L^2$-quaternionic-valued functions on $\Hqr$, and next extend the obtained expressions to the real line.

Notice that $\R$ is a negligible Borel measurable set with respect to the gaussian measure on $\Hq$, and therefore
\begin{align}\label{integration}
\int_{\Hq} f(q) e^{-|q|^2}d\lambda(q) & = \int_{\Hqr} f(q) e^{-|q|^2}d\lambda(q) \nonumber
\\& =  \int_{\R^{+*}\times ]0,2\pi[\times \Sq} f(re^{I\theta}) e^{-r^2} rdrd\theta d\sigma(I_q),
\end{align}
where $dr$ (resp. $d\theta$) denotes the Lebesgue measure on positive real line (the unit circle) and $d\sigma(I)$
stands for the standard area element on $\Sq$. This observation will be used systematically when discussing square integrability of the appropriate  extension on the whole $\Hq$.

\subsection{$\mathcal{C}^\infty$-right-eigenvalue problem.}

Let $\mu$ be a fixed quaternionic number and consider the right eigenvalue problem
$\Boxu_{q}f= f \mu$ for $f$ belonging to the right quaternionic vector space $\mathcal{C}^\infty(\Hq)$ of all quaternionic-valued functions that are $\mathcal{C}^\infty$ on the whole $\Hq\simeq \R^4$.
Thus, associated to $\mu$, we define the $\mathcal{C}^\infty$-eigenspace
\begin{align}\label{smootheigenspace}
\mathcal{E}^\infty_\mu(\Hq,\Boxu_{q}) := \left\{ f\in \mathcal{C}^\infty(\Hq); \,  \Boxu_{q}f= f \mu\right\}.
\end{align}
Notice for instance that $\mathcal{E}^\infty_\mu(\Hq,\Boxu_{q}) $ is not necessarily a  quaternionic right vector space. %, unless $\mu$ is real.
But, it is a $\mathcal{C}_\mu$-right vector space, where $\mathcal{C}_\mu:=\{ p\in \Hq, \, p\mu=\mu p\}$ is  the set of all quaternion numbers commuting with $\mu$. We have $\mathcal{C}_\mu = \Hq$ when $\mu$ is real and $\mathcal{C}_\mu $ is  $\C_\mu$ otherwise.

The first main result of this section concerns the explicit characterization of the elements of $\mathcal{E}^\infty_\mu(\Hq,\Box_{q})$.
Such description involves the Kummer's function defined by  %\cite{Rainville}
\begin{align}\label{NewHypergeometricFct}
M\left(  \begin{array}{c} a  \\ c \end{array} \bigg | x \right) = \sum_{j=0}^\infty \frac{(a)_j}{(c)_j} \frac{x^j}{j!} ,
\end{align}
for given $a \in \Hq$ and $x,c\in \R$, $c\ne 0,-1,-2,\cdots$, 
where $(a)_j$ denotes the rising factorial $(a)_j= a(a+1) \cdots (a+j-1)$ with $(a)_0=1$.
Namely, we have

\begin{theorem}\label{thm:CInftydescription}
	A $\mathcal{C}^\infty$-quaternionic-valued function $f$ on $\Hq$ is a solution of $\Boxu_{q}f= f \mu$ on $\Hqr$ if and only if it can be expanded
as
	\begin{align}\label{leftEigenCinfty}
	f(|q|e^{I_q\theta})=\sum_{j\in \Z} q^{(1+\sgn(j))\frac {|j|}2}\bq^{(1-\sgn(j))\frac {|j|}2}  \,
	M\left(  \begin{array}{c} -\mu-(1-\sgn(j))\frac j2 \\ |j|+1   \end{array} \bigg | |q|^{2} \right) \gamma_{\mu,j}^I
	\end{align}
	for some quaternionic-valued functions $I_q \mapsto \gamma^{I}_{\mu,j}$ on $\Sq$ with values in $\mathcal{C}_\mu$. Here $\sgn$ denotes the signum function.
\end{theorem}

\begin{proof} Let $f: \Hq \longrightarrow \Hq$ be a  $\mathcal{C}^\infty$-quaternionic-valued function which is solution of $\Boxu_{q}f= f \mu$ on $\Hqr$. Then, $\widetilde{f}=f|_{\Hqr}$ satisfies $\Boxd_{q}\widetilde{f}= \widetilde{f} \mu$, where $\Boxd_{q} $ denotes the restriction
	of the slice differential operator $\Boxu_{q} $ in \eqref{Laplaceian} to $\Hqr$.
	It takes the form
	\begin{equation}\label{LaprealcoordH1}
	\Boxd_{q} = -\frac{1}{4} \left( \frac{\partial^2}{\partial x^2} + \frac{\partial^2}{\partial y^2} \right)
	+ \frac{1}{2} \left(x\frac{\partial}{\partial x} + y\frac{\partial}{\partial y} \right)  +
	\frac{I}{2} \left( x\frac{\partial}{\partial y} - y\frac{\partial}{\partial x} \right).
	\end{equation}
	Its expression in polar coordinates $q=re^{I\theta}$, with $r>0$,
	$0\leq \theta \leq 2\pi$ and $I\in \mathbb{S}$, reads
	\begin{equation*}\label{1pc}
	\Boxd_{q}=-\dfrac{1}{4}\left(\dfrac{\partial^{2}}{\partial r^{2}}+\left[\dfrac{1}{r}-2r\right]\dfrac{\partial}{\partial r}+\dfrac{\partial^{2}}{r^{2} \partial \theta^{2}}-2I\dfrac{\partial}{\partial \theta}\right)
	\end{equation*}
	and its action, on any $\mathcal{C}^\infty$ function $(r,\theta) \longrightarrow e^{I j\theta}  a^{I}_{j}(r)$ on $ [0,2\pi[\times  [0,+\infty[$, is given by
	\begin{equation}\label{action}
	\Boxd_{q}( e^{I j\theta}  a^{I}_{j}(r)) =-\dfrac{e^{I n\theta}}{4r^2} \left[r^2\dfrac{\partial^{2}}{\partial r^{2}}+ (1-2r^2)r \dfrac{\partial}{\partial r}
	+ (2jr^2 -j^2)\right]   a^{I}_{j}(r).
	\end{equation}
	Since $\widetilde{f}\in\mathcal{C}^{\infty}(\Hqr)$ and its restriction $\widetilde{f}(re^{I\theta})$ to $\C_I$ is in addition periodic with respect to $\theta$, one can expand it in Fourier series as %that for every fixed $r\in [0,+\infty[$ and $I\in \mathbb{S}$, we have
	\begin{equation}\label{3}
	\widetilde{f}(re^{I\theta})=\sum_{j\in \Z}e^{I j\theta} a^{I}_{j}(r),
	\end{equation}
	where the functions $(r,I) \longmapsto a^{I}_{j}(r)$ are $\mathcal{C}^{\infty}$ on $[0,+\infty[\times \Sq$.
	Therefore, by inserting \eqref{3} in the right-eigenvalue problem $\Boxd_{q} \widetilde{f} = \widetilde{f} \mu$ taking into account \eqref{action},
    we see that
    \begin{equation}\label{1a}
	\left[r^2\dfrac{\partial^{2}}{\partial r^{2}}+ (1-2r^2)r \dfrac{\partial}{\partial r}
	+ (2jr^2 -j^2)\right]   a^{I}_{j}(r) =    -  4r^2 a^{I}_{j}(r)  \mu
	\end{equation}
	holds for every integer $j$ and fixed $r$ and $I$.
	Now, by the changes of variable $t=r^2>0$ and of function $a^{I}_{j}(r) = t^\alpha b_{j}(t,I)$, we get
   \begin{equation}\label{1b}
	t b_{j}^{''}(\cdot,I)  + (2\alpha +1 -t) b_{j}^{'}(\cdot,I)   + \dfrac{1}{t} \left(\alpha-\dfrac{j}{2}\right)\left(\alpha+\dfrac{j}{2} -t\right) b_{j}(\cdot,I) = - b_{j}(\cdot,I) \mu .
	\end{equation}
	For the ansatz $\alpha=|j|/2$, we recognize the left-quaternionic version of the confluent hypergeometric differential equation %\textcolor[rgb]{1.00,0.00,0.00}{\cite{xx}}
	\begin{equation}\label{6eq}
	t b_{j}^{''}(\cdot,I) + (c -t) b_{j}^{'}(\cdot,I)    =  b_{j}(\cdot,I) a
	\end{equation}
	satisfied by $b_{j}(\cdot,I) $ on $]0,+\infty[$, with $c = |j| +1$ and $a= - \mu - j \chi_{\Z^-}(j)=- \mu  - \left(1-\sgn(j)\right)\frac j2  \in \Hq$.
	Its first solution is given by the Kummer's function
	$ M\left(  \begin{array}{c} a  \\ c \end{array} \bigg | t \right)$, for $c=|j|+1$ being a positive integer. The second (linearly independent) solution is given by the Tricomi's logarithmic function \cite[p. 21]{Tricomi1960} (see also \cite[p. 504]{AbStegun1972})
	\begin{align*}
	U\left(  \begin{array}{c} a  \\ |j|+1 \end{array} \bigg | t \right) & := \frac{(|j|-1)!}{\Gamma(a)} S^a_j(t) + \frac{(-1)^{|j|+1}}{|j|! \Gamma(a-|j|)}
	\left\{ M\left(  \begin{array}{c} a  \\ |j|+1 \end{array} \bigg | t \right) \ln t \right. \\
	& + \left. \sum_{k=0}^{+\infty} \frac{(a)_k}{(|j|+1)_k}(\psi(a+k)-\psi(1+k)-\psi(|j|+1+k)) \frac{t^{k}}{k!}\right\} ,
	\end{align*}
	where $\psi(x)$ denotes the logarithmic derivative of the gamma function, $\psi(x)=\Gamma'(x)/\Gamma(x)$, and $S^a_j(t)$ is the finite sum given by
	$$ S^a_j(t):= \sum_{k=0}^{+\infty}\frac{(a-|j|)_k}{(1-|j|)_k}\frac{t^{k-|j|}}{k!}$$
	and interpreted as $0$ when $j=0$.
	 Thus, the only solution of \eqref{6eq} that can be extended to a A $\mathcal{C}^\infty$ function at %the regular singular % nonessential singularity point
	$t=0$ is given by  
    $$ b_{j}(t,I)=  M\left(   \begin{array}{c} -\mu - \left(1-\sgn(j)\right)\frac j2\\ |j|+1 \end{array} \bigg | t \right)
	\gamma_{\mu,j}^I$$
	for some quaternionic constants $\gamma_{\mu,j}^I\in \mathcal{C}_\mu$ (viewed as functions on $\Sq$).
	Therefore, the corresponding $f$, whose restriction to $\Hqr$ are solutions of the right-eigenvalue problem $\Boxd_{q} \widetilde{f} = \widetilde{f} \mu$, are given by
	\begin{align*}
	f(re^{I\theta})=\sum_{j\in \Z} r^{|j|}e^{jI\theta}\, M\left(  \begin{array}{c}  -\mu - \left(1-\sgn(j)\right)\frac j2\\ |j|+1 \end{array} \bigg | r^{2} \right) \gamma_{\mu,j}^I.
	\end{align*}
	They can rewritten as in \eqref{leftEigenCinfty}.
	Such expression is well-defined as a $\mathcal{C}^\infty$ function on the whole $\Hq$.
\end{proof}

\begin{remark}
	The extension of the solution of differential equation \eqref{6eq} at the regular singular point $0$ corresponds to the extension of the solution %, given by \eqref{leftEigenCinfty},
	of the right-eigenvalue problem $\Boxd_{q} f = f \mu$ on $\Hqr$ to the whole $\Hq$.
\end{remark}

\begin{remark}
	The quaternionic $\mathcal{C}_\mu$-right-vector space $\mathcal{E}^\infty_\mu(\Hq,\Boxu_{q})$ is generated by the functions
	\begin{align}\label{elementaryFctmucst}
	\psi_{\mu,j}(q)  := q^{(1+\sgn(j))\frac {|j|}2}\bq^{(1-\sgn(j))\frac {|j|}2}  \,
	M\left(  \begin{array}{c} -\mu-(1-\sgn(j))\frac j2 \\ |j|+1   \end{array} \bigg | |q|^{2} \right)
	\end{align}
for varying $j\in \Z$.	The expansion of any $f\in \mathcal{E}^\infty_\mu(\Hq,\Boxu_{q})$ in terms of $\psi_{\mu,j}(q)$ involves slice right coefficients $\gamma_{\mu,j}^I \in \mathcal{C}_\mu$.
\end{remark}

\subsection{$L^2$-right-eigenvalue problem.}
In the sequel, we are interested in giving a concrete description of $L^2$-eigenspaces of the right-eigenvalue problem $\Boxu_{q}f= f \mu$.
To this end, we define
\begin{equation}\label{muL2eigen}
\mathcal{F}_{\mu}^{2}:=  \left\{ f\in L^{2}(\Hq;e^{-|q|^{2}}d\lambda); \,  \Boxu_{q} f = f \mu \right\},
\end{equation}
as well as
\begin{equation}\label{muL2eigend}
\widetilde{\mathcal{F}_{\mu}^{2}} :=  \left\{ f\in L^{2}(\Hqr;e^{-|q|^{2}}d\lambda); \,  \Boxd_{q} \widetilde{f} = \widetilde{f} \mu \right\},
\end{equation}
where $L^{2}(\Hq;e^{-|q|^2}d\lambda)$ denotes the right quaternionic Hilbert space of all quaternionic-valued square integrable functions on $\Hq$ with respect to the inner product
\begin{align}\label{SP-full}
\scal{ f, g}_{\Hq} : =\int_{\Hq}\overline{f(q)}g(q) e^{-|q|^{2}}d\lambda(q)
\end{align}
with $d\lambda(q)= dx_0dx_1dx_2dx_3$ being the Lebesgue measure on $\Hq\simeq \R^4$. We define in a similar way $L^{2}(\Hqr;e^{-|q|^2}d\lambda)$ and $\scal{ \widetilde{f} , \widetilde{g} }_{\Hqr}$.
Thus, the following lemmas are fundamentals for our investigation of the $L^2$-eigenspaces $\mathcal{F}_{\mu}^{2}$.

\begin{lemma} With the same notations as above, we have the following results.
	\begin{enumerate}
		\item [(i)] It holds
		 $$ Spec_{L^{2}(\Hq;e^{-|q|^2}d\lambda)}(\Boxu_{q}) \subset Spec_{L^{2}(\Hqr;e^{-|q|^2}d\lambda)}(\Boxd_{q}),$$
		 where $Spec$ denotes the spectrum of the prescribed operator.
		\item [(ii)] The space $\mathcal{F}_{\mu}^{2}$ is a $L^2$-subspace of $\mathcal{E}^\infty_\mu(\Hqr,\Boxd_{q})$ and we have
		\begin{equation}\label{mu}
		\mathcal{F}_{\mu}^{2}\subset \widetilde{\mathcal{F}_{\mu}^{2}} =  L^{2}(\Hqr;e^{-|q|^{2}}d\lambda) \cap \mathcal{E}^\infty_\mu(\Hqr,\Boxd_{q}).
		\end{equation}
	\end{enumerate}
\end{lemma}

\begin{proof}
	The first assertion holds true since for every $f \in L^{2}(\Hq;e^{-|q|^2}d\lambda)$, we have $\widetilde{f} \in L^{2}(\Hqr;e^{-|q|^2}d\lambda)$ with
	$\norm{f}_{\Hq} = \norm{\widetilde{f}}_{\Hqr}$. The second assertion is an immediate consequence of the ellipticity of $\Boxd_{q}$ seen as a second order differential operator on $\R\times \R^*$ (see \cite{Evans1998,Morrey2008}).
\end{proof}

The second key lemma concerns the elementary functions
$$\varphi_{\mu,j}(q) := \psi_{\mu,j}(q) \alpha_{j}^I $$
associated to given $\alpha_{j}^I\in \mathcal{C}_\mu$, where $q=x+Iy\in \Hq$ and $\psi_{\mu,j}$ are as in \eqref{elementaryFctmucst}.

\begin{lemma} \label{keyLem} The following results hold true.
	\begin{enumerate}
		\item [(i)] The functions $\varphi_{\mu,j}$ are pairwisely orthogonal in the sense that
		$ \scal{\varphi_{\mu,j},\varphi_{\mu,k}}  = 0$ whenever $j\ne k$.
		\item [(ii)]  The functions $\varphi_{\mu,j}$ belong to $L^{2}(\Hq;e^{-|q|^{2}}d\lambda)$ if and only if $\mu_j=\mu+j$ is a nonnegative integer.
		\item [(iii)] Let $\mu_j=0,1,2,\cdots$. Then, the square norm of $\varphi_{\mu_j,j}$ in $L^{2}(\Hq;e^{-|q|^{2}}d\lambda)$ is given by
		\begin{align}\label{normelementaryFctmucst}
		\norm{\varphi_{\mu_j,j}}^2_{\Hq} = \pi \dfrac{\mu_j!(|j|!)^2}{(\mu_j+j)!}  \int_{\Sq} | \alpha_{j}^I|^2 d\sigma(I).
		\end{align}
	\end{enumerate}
\end{lemma}

\begin{proof}
	The first assertion follows by direct computation using polar coordinates, $q=re^{I\theta}$.
	Indeed, in these coordinates, the Lebesgue measure $d\lambda$ becomes the product of the standard
	 measures $rdr$ on $\R^{+}$ and the Lebesgue measure $d\theta$ on the unit circle times the standard area element $d\sigma(I)$ on $\Sq$,
	the two-dimensional sphere of imaginary units in $\Hq$. Therefore,  for every $\alpha_{j}^I\in \Hq$, we have % by the Fubini's theorem
	\begin{align}
	\scal{\varphi_{\mu_j,j} ,\varphi_{\mu_k,k} } & =
	\int_{\Hqr}  \overline{\psi_{\mu_j,j}(q) \alpha_{j}^I}\psi_{\mu_k,k}(q)  \alpha_{k}^I  e^{-|q|^2} d\lambda(q) \nonumber\\
	&=\int_{0}^\infty r^{|j|+|k|+1} \int_{\Sq} \overline{  \alpha_{j}^I } R_{j,k}(I)   \alpha_{k}^I e^{-r^{2}}  d\sigma(I) dr ,
	\end{align}
	where $R_{j,k}(I)$ stands for
	$$ R_{j,k}(I) :=   \overline{  M\left(  \begin{array}{c} -\mu_j \\ |j|+1 \end{array} \bigg | r^{2} \right) }
	\left(\int_{0}^{2\pi} e^{ (k-j)I\theta} d\theta \right)   M\left( \begin{array}{c} -\mu_k \\ |k|+1 \end{array}  \bigg | r^{2} \right) .$$
	The use of the well-known fact $\int_{0}^{2\pi} e^{(k-j)I\theta} d\theta=2\pi \delta_{j,k}$ completes our check of $(i)$. Now, by the change of variable $t=r^{2}$ we obtain
	\begin{align}
	\scal{\varphi_{\mu_j,j},\varphi_{\mu_k,k}} &=  \pi  \left(\int_{\Sq} | \alpha_{j}^I|^2 d\sigma(I)\right) \left(\int_{0}^\infty t^{j} \left|  M\left(  \begin{array}{c} -\mu_j \\ |j|+1 \end{array} \bigg | t \right) \right|^2  e^{-t}  dt \right)\delta_{j,k}. \label{elemIntegral}
	\end{align}
	Therefore, to prove the second assertion, we make use of the asymptotic behavior
	$$ M\left( \begin{array}{c} a\\ c \end{array}  \bigg | t \right)  \sim  \frac{e^{t}t^{a-c}}{\Gamma(a)} $$
	for $t$ large enough and $a \ne 0,-1,-2, \cdots$, that follows from the Poincar\'e-type expansion \cite[Section 7.2]{Temme1996} % See Temme (1996b, Â§7.2) or Olver (1997b, pp. 256â258).
	$$M\left( \begin{array}{c} a\\ c \end{array}  \bigg | t \right)  \sim  \frac{e^{t}t^{a-c}}{\Gamma\left(a\right)}%
	\sum_{k=0}^{\infty}\frac{{\left(1-a\right)_{k}}{\left(c-a\right)_{k}}}{k!}t^{-k}.$$
	Indeed, if $\mu_j\ne 0,1,2,\cdots$, then the nature of the integral involved in the right-hand side of \eqref{elemIntegral} is equivalent to
	$$  \frac{1}{|\Gamma(-\mu_j)|^2} \int_{0}^\infty   t^{-(2\Re(\mu_j)+|j|+2)} e^{t}   dt$$ which is clearly divergent. % for $j$ large enough.
	Thus, we necessarily have $\mu_j=0,1,2,\cdots$. In this case, the involved Kummer's function is the generalized Laguerre polynomial
	(\cite[Eq. (1), p. 200]{Rainville71})
	\begin{align}\label{HyperLaguerre}
	M\left( \begin{array}{c} -\mu_j \\ |j|+1 \end{array}\bigg | t \right)&=\dfrac{\mu_j!}{(|n|+1)_{\mu_j}}L_{\mu_j}^{(j)}(t)
	\end{align}
	which satisfies the following orthogonality property \cite[Eq. (4), p. 205 - Eq. (7), p. 206]{Rainville71}
	\begin{align}\label{OrthLaguerre}
	\int_{\mathbb{R}^{+}}L^{(\alpha)}_{j}(t)L^{(\alpha)}_{k}(t) t^{\alpha }e^{-t} dt
	= \frac{\Gamma(\alpha +j+1)}{\Gamma(j+1) } \delta_{j,k}.
	\end{align}
	More precisely, starting from \eqref{elemIntegral}, the explicit computation yields
	\begin{align*}
	\norm{\varphi_{\mu_j,j}}^2_{\Hq} & =
	\pi \left( \dfrac{\mu_j!}{(|j|+1)_{\mu_j}}\right)^2 \left(\int_{0}^\infty (L_{\mu_j}^{(j)}(t))^2  t^{j}  e^{-t}  dt\right) \times \left(\int_{\Sq} | \alpha_{j}^I |^2 d\sigma(I)\right)
\\&
	= \pi \dfrac{\mu_j!(j!)^2}{(\mu_j+j)!}  \left( \int_{\Sq} | \alpha_{j}^I|^2 d\sigma(I)\right) .
	\end{align*} 
	This completes the proof of $(ii)$ and $(iii)$.
\end{proof}

\begin{remark}
	If $\mu$ is a fixed nonnegative integer $\mu=n$, then $\psi_{\mu_j,j} \alpha_{j}^I$ belongs to $L^{2}(\Hq;e^{-|q|^2}d\lambda)$ if and only if $j\geq -n$, unless the corresponding $\alpha_{j}^I$ is zero. In this case, the square norm of $\psi_{n,j}$ (in \eqref{elementaryFctmucst})
is given by
	\begin{align} \label{normeElemFct}
	\norm{\psi_{n,j}}^2_{\Hq}  = \pi \dfrac{n!(j!)^2}{(n+j)!}  Area(\Sq) ,
	\end{align}
where $Area(\Sq)$ denotes the surface area of $\Sq$.
\end{remark}

The next result shows in particular that the spectrum of $\Boxu_{q}$ acting  $L^{2}(\Hq;e^{-|q|^{2}}d\lambda)$ is purely discrete and reduces to the quantized eigenvalues known as Landau levels.

\begin{theorem}\label{thm:description}
	The space $\mathcal{F}_{\mu}^{2}$ is nontrivial if and only if $\mu=n=0,1,2, \cdots.$ In this case, a nonzero quaternionic-valued function $f$ belongs to $\eigendn(\Hq)$ if and only if it can be expanded as 
	\begin{align}\label{expansionmonomials}
	f(q)=\sum_{j=-n}^{+\infty} q^{j}   \,  M\left( \begin{array}{c} -n \\ |j|+1 \end{array}  \bigg | |q|^{2}  \right)  \Cj,
	\end{align}
	where the quaternionic constants $\Cn$ satisfy the growth condition
	\begin{align}\label{Growth}
	%\sum_{n=0}^{+\infty}\gamma(n,m)|C_{n}|^{2}<+\infty, \norm{f}^{2} = \pi m!\sum_{n=0}^{+\infty}\dfrac{n!^{2}}{(n+m)!}|C_{n}|^{2}
	\norm{f}^{2}_{\Hq}  = \pi \sum_{j=-n}^{+\infty}
	\frac{n!(j!)^2}{(n+j)!} \left(\int_{\Sq} |\Cj|^2  d\sigma(I)\right) <+\infty.
	\end{align}
\end{theorem}

\begin{proof}%[Proof of Theorem \ref{thm:description}]
	Fix $\mu \in \Hq$ and assume that there is a nonzero $f\in L^{2}(\Hq;e^{-|q|^{2}}d\lambda)$ solution of $\Boxu_{q}f= f\mu $.
	Then, the realization \eqref{mu} and the proof of Theorem \ref{thm:CInftydescription} show that $\widetilde{f}:=f|_{\Hqr}$ admits the expansion
	\begin{align*}
	\widetilde{f}(|q|e^{I_q\theta}) = \sum_{j\in\Z} \psi_{\mu_j,j}(q) \gamma_{\mu,j}^I .
	 \end{align*}
	 The orthogonality of the $(\psi_{\mu_j,j})_j$ (see (i) of Lemma \ref{keyLem}) infers that
	\begin{align*}
	\norm{\widetilde{f}}^{2}_{\Hqr} &
	= \sum_{j\in\Z}   \norm{\psi_{\mu_j,j}  \gamma_{\mu,j}^I }^2_{\Hq}
	\\&= \frac{\pi}{ Area(\Sq)}\sum_{j\in\Z}   \left(\int_{\Sq} |\gamma_{\mu_j,j}^I|^2 d\sigma(I)\right) \norm{\psi_{\mu_j,j} }^2_{\Hq}.
	\end{align*}
	Therefore, since the nonzero function $f$ belongs to $\mathcal{F}_{\mu}^{2}$, we have necessarily $\norm{\psi_{\mu_j,j}}^2_{\Hq} $ is finite for every $j$ such that
	$$ \int_{\Sq} |\gamma_{\mu,j}^I|^2 d\sigma(I)  \ne 0.$$
	 Now, (ii) of Lemma \ref{keyLem} readily implies that $\mu$ is necessary of the form $\mu=n=0,1,2, \cdots$, and $j\geq -n$.
	In such case, the $\gamma^{I}_{\mu,j}=:\Cj $ are arbitrary in $\Hq=\mathcal{C}_\mu$ for $\mu$ being real. Moreover, we have
	\begin{align*}
	\norm{f}^{2}_{\Hq} &   = \pi \sum_{j=-n}^{+\infty}  \dfrac{n!(j!)^2}{(n+j)!}  \int_{\Sq} | \Cj |^2 d\sigma(I) .
	\end{align*}
	This yields the growth condition \eqref{Growth} and the proof is completed.
\end{proof}

The following result describes the fact that the elements of $\eigendn$ can be expanded as series of the quaternionic Hermite polynomials $H^Q_{j,n}(q,\overline{q})$.

\begin{corollary}\label{cor:expansionHermite}
	The space $\eigendn$ contains the quaternionic Hermite polynomials $(H^Q_{n+j,n})_j$ defined by \eqref{20}.
	Moreover, every element $f$ belonging to $\eigendn $ can be expanded as
	\begin{align}\label{expansionHermite}
f(q) &=\sum_{j=-n}^{+\infty} \dfrac{(-1)^{j}j!}{(n+j)!} H^Q_{n+j,n} (q,\overline{q}) \Cj
	\end{align}
	for some slice quaternionic constants $\Cn$ displaying the growth condition \eqref{Growth}.
\end{corollary}

\begin{proof}
	This lies in the fact that the involved confluent hypergeometric function is connected to the quaternionic Hermite polynomials through
	\begin{align}\label{HypHermite}
	q^{j} M\left( \begin{array}{c} -n \\ |j|+1 \end{array}  \bigg | |q|^{2}  \right) = \dfrac{(-1)^{n}j!}{(n+j)!}  H^Q_{n+j,n} (q,\overline{q}),
	\end{align}
	see Lemma 3.2 in \cite{ElHamyani2018}.
	Therefore, the expression of $f(q)$ given through \eqref{expansionmonomials} reduces further to \eqref{expansionHermite}
	with the same growth condition \eqref{Growth}.
\end{proof}

\subsection{Connection to S-polyregular Bargmann spaces of first kind.}
By Corollary \ref{cor:expansionHermite}, the space $\eigendn$ can be realized as the space of the convergent series
$$\sum_{j=0}^{+\infty} \dfrac{(-1)^{n}j!}{j!} H^Q_{j,n} (q,\overline{q}) C_j(I_q)$$
on $\Hq$, where $(C_j(I_q))_j$ are certain quantities in $\in\Hq$ such that
$$\pi \sum_{j=-n}^{+\infty}
\frac{n!(j!)^2}{(n+j)!} \left(\int_{\Sq} |C_j(I_q)^2  d\sigma(I)\right) <+\infty.$$
It reduces further to $\srddn$ when the $C_j(I_q)$ are assumed to be constant functions on $\Sq$, $C_j(I_q)=C_j$.
In particular, by taking $n=0$, the previous growth condition reads simply as
\begin{align*}
\pi Area(\Sq)\sum_{j=0}^{+\infty} j! |C_j|^2   <+\infty.
\end{align*}
Comparing this to the sequential characterization of the slice hyperholomorphic Bargmann space $ \mathcal{F}^{2}_{slice}$ given by Proposition 3.11 in \cite{AlpayColomboSabadiniSalomon2014},
we see that   $\mathcal{F}^{2}_{slice} \subset  \mathcal{F}_{0}^{2}$.

\section{Concluding Remarks: Full S-polyregular Bargmann spaces}\label{gbs2}

Motivated by Theorem 4.2 in \cite{ElHamyani2018} asserting that the quaternionic Hermite polynomials $(H^Q_{j,k})_{j,k}$ form a slice basis of the Hilbert space  $L^{2}(\Hq; e^{-|q|^2}d\lambda)$,
equipped with the scalar product
\begin{align}\label{SP-full}
\scal{ f, g}_{\Hq} =\int_{\Hq}\overline{f(q)}g(q) e^{-|q|^{2}}d\lambda(q),
\end{align}
we define $\mathcal{SR}_{n,full}^{2}$ to be the space of S-polyregular functions (of level $n$) spanned by the quaternionic Hermite polynomials $H^Q_{j,n}$, for varying $j=0,1,2,\cdots$, and belonging to $L^{2}(\Hq; e^{-|q|^2}d\lambda)$.
Then, we have
$$ \scal{ f, g}_{\Hq} = \int_{\Sq} \scal{ f, g}_{\C_I} d\sigma(I)$$
and subsequently, the space $\mathcal{SR}_{n,full}^{2}$ can be described as the right quaternionic vector space consisting of the convergent series
$$ \sum_{j=0}^{+\infty} H^Q_{j,n}(q,\overline{q})  C_j(I_q)$$
on $\Hq$, with $C_j:\Sq \longrightarrow \Hq$, and such that
$$ \pi n! \sum_{j=0}^\infty  \int_{\Sq} |C_j(I)|^2 d\sigma(I) <+\infty.$$
This is exactly the sequential characterization of $L^2$-eigenspace
$\eigendn$.  The particular case of $n=0$ corresponds to the full hyperholomorphic Bargmann space
\begin{align} \label{full}
\mathcal{F}_{full}^{2} :=\mathcal{SR}\cap L^{2}(\Hq;e^{-|q|^2}d\lambda)
\end{align}
defined as the right quaternionic Hilbert space of all slice regular functions that are $e^{-|q|^2}d\lambda$-square integrable on $\Hq$. This lies on the fact that $ \mathcal{F}^{2}_{full}$ is the space of functions
$f(q) = \sum\limits_{j=0}^\infty q^j \Cj$ satisfying
$$ \norm{f}^2_{\Hq} = \pi\sum_{j=0}^{+\infty} j! \left(\int_{\Sq} |\Cj|^2  d\sigma(I)\right) < +\infty.$$
More generally, it is not difficult to prove that the spaces $\mathcal{SR}_{n,full}^{2} $ are right quaternionic Hilbert spaces. We call them here the full S-polyregular Bargmann spaces of second kind of level $n$. The quaternionic Hermite polynomials $H^Q_{j,n}$, for varying $j=0,1,2,\cdots$, constitute an orthogonal "slice" basis of it. 

\quad

{\bf\it\bf  Acknowledgments.}  The authors would like to thank the anonymous referees for helpful comments and suggestions.
%
%The present version is an improved version of an earlier version in ArXiv (see \cite{ElHamyaniG2017}) which is a part of Ph.D thesis of the second-named author (defended by January 2017).
%Another part of this investigation was completed during the third-named author's visit to Departimento di Matematica of Politecnico di Milano May - June 2017.  He would like to express his gratitude to Prof. I.M. Sabadini for invitation and hospitality. 
The research work of A.G. was partially supported by a grant from the Simons Foundation.

\end{document}